\documentclass{amsart}
\usepackage{amssymb}

\input xy
\xyoption{all}

\newcommand{\dom}{\mathrm{dom}}
\newcommand{\edim}{\mbox{\rm e-dim}}
\newcommand{\F}{\mathcal F}
\newcommand{\G}{\mathcal G}

\newcommand{\GGamma}{\bar\Gamma}
\newcommand{\pr}{\mathrm{pr}}
\newcommand{\IR}{\mathbb R}
\newcommand{\IZ}{\mathbb Z}
\newcommand{\IT}{\mathbb T}
\newcommand{\IC}{\mathbb C}
\newcommand{\IQ}{\mathbb Q}
\newcommand{\e}{\varepsilon}
\newcommand{\IN}{\mathbb N}
\newcommand{\A}{\mathcal A}
\newcommand{\C}{\mathcal C}

\newcommand{\w}{\omega}
\newcommand{\lrN}{\overset{\leftrightarrow}{\mathcal N}}
\newcommand{\rN}{\overset{\rightarrow}{\mathcal N}}
\newcommand{\diam}{\mathrm{diam}}
\newcommand{\DD}{\mathcal{D}}
\newcommand{\conv}{\mathrm{conv}}
\newcommand{\U}{\mathcal U}
\newcommand{\V}{\mathcal V}
\newcommand{\W}{\mathcal W}
\newcommand{\St}{\mathcal{S}t}
\newcommand{\Comp}{\mathsf{Comp}}
\newcommand{\CompEpi}{\mathsf{CompEpi}}
\newcommand{\ANE}{\mathsf{ANE}}
\newcommand{\setmap}{\multimap}

\newtheorem{theorem}{Theorem}[section]

\newtheorem{corollary}[theorem]{Corollary}
\newtheorem{problem}[theorem]{Problem}
\newtheorem{claim}[theorem]{Claim}
\newtheorem{lemma}[theorem]{Lemma}
\theoremstyle{definition}

\title{Dimension of graphoids of rational vector-functions}
\author{Taras Banakh and Oles Potyatynyk}
\address{Instytut Matematyki, Jan Kochanowski University, Kilece, Poland and \newline
Department of Mathematics, Ivan Franko National University of Lviv, Ukraine}
\email{tbanakh@yahoo.com}
\address{Department of Mathematics, Ivan Franko National University of Lviv, Ukraine}
\email{oles2008@gmail.com}

\begin{document}
\begin{abstract} Let $\F\subset\IR(x,y)$ be a countable family of rational functions of two variables with real coefficients. Each rational function $f\in\F$ can be thought as a continuous function $f:\dom(f)\to\bar\IR$ taking values in the projective line $\bar\IR=\IR\cup\{\infty\}$  and defined on a cofinite subset $\dom(f)$ of the torus $\bar\IR^2$. Then the family $\F$ determines a continuous vector-function $\F:\dom(\F)\to\bar\IR^\F$ defined on the dense $G_\delta$-set $\dom(\F)=\bigcap_{f\in\F}\dom(\F)$ of $\bar\IR^2$. The closure $\bar\Gamma(\F)$ of its graph $\Gamma(\F)=\{(x,f(x)):x\in\dom(\F)\}$ in $\bar\IR^2\times\bar\IR^\F$ is called the {\em graphoid} of the family $\F$. We prove the graphoid $\bar\Gamma(\F)$ has topological dimension $\dim(\bar\Gamma(\F))=2$. If the family $\F$ contains all linear fractional transformations $f(x,y)=\frac{x-a}{y-b}$ for $(a,b)\in\IQ^2$, then the graphoid $\bar\Gamma(\F)$ has cohomological dimension $\dim_G(\bar\Gamma(\F))=1$ for any non-trivial 2-divisible abelian group $G$. Hence the space $\bar\Gamma(\F)$ is a natural example of a compact space that is not dimensionally full-valued and by this property resembles the famous Pontryagin surface.
\end{abstract}
\maketitle

\section{Introduction}
Let $X,Y$ be topological spaces and $f:\dom(f)\to Y$ be a function defined on a subset $\dom(f)\subset X$. Such a function $f$ will be called a {\em partial  function on $X$}. The closure $\GGamma(f)$ of the graph
$$\Gamma(f)=\{(x,f(x)):x\in\dom(f)\}$$ of $f$ in the Cartesian product $X\times Y$ will be called the {\em graphoid} of $f$. The graphoid $\GGamma(f)$ determines a multi-valued function $\bar f:X\setmap  Y$ assigning to each point $x\in X$ the (possibly empty) subset $\bar f(x)=\{y\in Y:(x,y)\in\GGamma(f)\}$.
It is clear that $\GGamma(f)$ coincides with the graph $\Gamma(\bar f)=\{(x,y)\in X\times Y:y\in\bar f(x)\}$ of the multi-valued function $\bar f:X\setmap  Y$. Also it is clear that $f(x)\in\bar f(x)$ for each $x\in\dom(f)$. The multi-valued function $\bar f$ is called {\em the graphoid extension} of the partial function $f$. The set  $\dom(\bar f)=\{x\in X:\bar f(x)\ne\emptyset\}$ will be called the {\em domain} of $\bar f$. If the space $Y$ is compact, then the projection $\pr_X:\Gamma(\bar f)\to X$ is a perfect map \cite[3.7.1]{Eng}, which implies that the multi-valued map $\bar f$ is upper semi-continuous in the sense that for any open subset $U\subset Y$ the preimage $\bar f^{-1}_\subset=\{x\in X:\bar f(x)\subset U\}$ is open in $X$.

In this paper we shall study topological properties of the graphoids of rational vector-functions.
By a {\em rational function} of $k$ variables
we understand a partial function $f:\dom(f)\to\bar\IR$ of the form $$f(x_1,\dots,x_k)=\frac{p(x_1,\dots,x_k)}{q(x_1,\dots,x_k)}$$ where $p$ and $q$ are two relatively prime polynomials of $k$ variables. The rational function $f=\frac{p}{q}$ is defined on the open dense subset
$$\dom(f)=\IR^k\setminus (p^{-1}(0)\cap q^{-1}(0))$$of $\bar\IR^k$ and takes its values in the projective real  line $\bar\IR=\IR\cup\{\infty\}$ (carrying the topology of one-point compactification of the real line $\IR$).

By $\IR(x_1,\dots,x_k)$ we denote the field of rational functions of $k$ variables with coefficients in the field $\IR$ of real numbers. Each rational function $f\in\IR(x_1,\dots,x_k)$ will be thought as a partial function defined on the open dense subset $\dom(f)$ of the $k$-dimensional torus $\bar\IR^k$ with values in the projective line $\bar\IR$.

 By a {\em rational vector-function} we understand any family $\F\subset\IR(x_1,\dots,x_k)$ of rational functions.

If $\F$ is countable, then the intersection $\dom(\F)=\bigcap_{f\in \F}\dom(f)$ is a dense $G_\delta$-set in $\bar\IR^k$.
So, $\F$ can be thought as a partial function
$$\F:\dom(\F)\to \bar\IR^F,\;\F:x\mapsto (f(x))_{f\in \F}.$$
Its graphoid $\bar\Gamma(\F)$ is a closed subset of the compact Hausdorff space $\bar\IR^k\times\bar\IR^\F$ and its graphoid extension $\bar \F:\bar\IR^k\setmap \bar\IR^F$ is an upper semi-continuous multi-valued function with $\dom(\bar \F)=\bar\IR^k$. For every $f\in\F$ the composition $\pr_f\circ\bar \F:\bar\IR^k\to\bar\IR$ of $\bar \F$ with the projection $\pr_f:\bar\IR^\F\to\bar\IR$, $\pr_f:(x_g)_{g\in \F}\mapsto x_f$, coincides with the graphoid extension $\bar f$ of the rational function $f$.

For uncountable families $\F\subset\IR(x_1,\dots,x_k)$ this approach to defining $\bar \F:\bar\IR^k\to\bar\IR^\F$ does not work properly as $\dom(\F)=\bigcap_{f\in \F}\dom(f)$ can be empty. This problem can be fixed as follows.

Let $\F$ be a family of partial functions $f:\dom(f)\to Y$ defined on subsets $\dom(f)$ of a topological space $X$. By the {\em graphoid extension} of $\F$ we understand the multi-valued function $\bar\F:X\setmap  Y$ assigning to each point $x\in X$ the set $\bar \F(x)$ of all points $y=(y_f)_{f\in \F}\in Y^\F$ such that for any neighborhood $O(x)\subset X$ of the point $x$, any finite subfamily $\mathcal E\subset \F$,  and neighborhoods $O(y_f)\subset Y$ of the points $y_f$, $f\in \mathcal E$, there is a point $x'\in O(x)\cap\bigcap_{f\in\mathcal E}\dom(f)$ such that $f(x')\in O(y_f)$ for all $f\in \mathcal E$.
The graph $$\Gamma(\bar\F)=\{(x,y)\in X\times Y^\F:y\in\F(x)\}$$of the multi-valued function $\bar \F$ is called the {\em graphoid} of the family $\F$. The set $\dom(\bar\F)=\{x\in X:\bar\F(x)\ne\emptyset\}$ is called the {\em domain} of $\bar\F$.

If the family $\F$ is empty, then $\bar Y^\F=Y^\emptyset$ is a singleton and the graphoid $\bar\Gamma(\F)$ coincides with $X\times Y^\emptyset$.

It can be shown that for any family of rational functions $\F\subset\IR(x_1,\dots,x_k)$ its graphoid extension $\bar \F:\bar\IR^k\setmap \bar\IR^F$ has the following properties:
\begin{enumerate}
\item $\bar \F$ is upper semi-continuous;
\item $\dom(\bar \F)=\bar\IR^k$;
\item for any  subfamily $\mathcal E\subset\F$ and the coordinate projection $\pr_{\mathcal E}:\bar\IR^\F\to\bar\IR^{\mathcal E}$ the composition $\pr_{\mathcal E}\circ\bar \F:\IR^k\to\bar\IR^{\mathcal E}$ coincides with the graphoid  extension $\bar{\mathcal E}$ of $\mathcal E$;
\item If $\dom(\F)=\bigcap_{f\in \F}\dom(f)$ is dense in $\IR^k$, then $\bar \F$ coincides with the graphoid extension of the partial function $\F:\dom(\F)\to\bar\IR^\F$, $\F:x\mapsto(f(x))_{f\in \F}$.
\end{enumerate}

In this paper we shall consider the following problem.

\begin{problem}\label{prob1} Given a family of rational functions $\F\subset\IR(x_1,\dots,x_k)$, study topological (and dimension) properties of the graphoid $\bar \Gamma(\F)\subset\bar \IR^k\times\IR^\F$ of $\F$.

 A precise question: Has $\bar\Gamma(\F)$ the topological dimension $\dim(\Gamma(\bar\F))=k$?
\end{problem}

This problem was motivated by the problem of studying the topological structure of the space of real places of a field of rational functions, posed in \cite{BG} and partly solved in \cite{MMO}, \cite{EO}. In this paper we shall answer Problem~\ref{prob1} for $k\le 2$.

In fact, the case $k=1$ is trivial: each rational function $f\in\IR(x)$ admits a continuous extension to $\bar\IR$ and can be thought as a continuous function $f:\bar\IR\to\bar\IR$. Then any family $\F\subset\IR(x)$ can be thought as a continuous function $\F:\bar\IR\to\bar\IR^\F$. Its graphoid extension $\bar \F$ coincides with $\F$. Consequently, the graphs $\bar \Gamma(\F)=\Gamma(\F)$ are homeomorphic to the projective real line $\bar\IR$ and hence, $\dim(\bar\Gamma(\F))=\dim(\bar\IR)=1$.

The case of two variables is much more difficult.
The following theorem is the main result of this paper and has a long and technical proof that exploits tools of Real and Complex Analysis, Algebraic Geometry, Algebraic Topology, Dimension Theory, General Topology, and Combinatorics. This theorem has been applied in \cite{BKMOP} for evaluating the dimension of the space of real places of some function fields.

\begin{theorem}\label{main1} For any family of rational functions  $\F\subset\IR(x,y)$ its graphoid $\bar\Gamma(\F)\subset\bar\IR^2\times\bar\IR^\F$ has covering topological dimension $\dim(\Gamma(\bar\F))=2$.
\end{theorem}

This theorem reveals only a part of the truth about the dimension of $\bar\Gamma(\F)$. The other part says that for sufficiently rich families $\F$ the graphoid $\bar\Gamma(\F)$ has cohomological dimension $\dim_G\Gamma(\bar \F)=1$ for any 2-divisible abelian group $G$! So, $\bar\Gamma(\F)$ is a natural example of a compact space which is not dimensionally full-valued. A classical example of this sort is the Pontryagin surface: a surface with glued M\"obius bands at each point of a countable dense set, see \cite[\S4.7]{ADim}.

The covering and cohomological dimensions are partial cases of the extension dimension \cite{ED} defined as follows. We say that the {\em extension dimension} of a topological space $X$ does not exceed a topological space $Y$ and write $\edim(X)\le Y$ if each continuous map $f:A\to Y$ defined on a closed subspace $A$ of $X$ can be extended to a continuous map  $\bar f:X\to Y$. By Theorem 3.2.10 of \cite{End}, a compact Hausdorff space $X$ has covering dimension $\dim(X)\le n$ for some $n\in\w$ if and only if $\edim(X)\le S^n$ where $S^n$ stands for the $n$-dimensional sphere.

On the other hand, for a non-trivial abelian group $G$, a compact topological space $X$ has cohomological dimension $\dim_G(X)\le n$ if and only if $\edim(X)\le K(G,n)$ where $K(G,n)$ is the Eilenberg-MacLane complex of $G$ (this is a CW-complex having all homotopy groups trivial except for the $n$-th homotopy group $\pi_n(K(G,n))$ which is isomorphic to $G$, see \cite[\S4.2]{Hat}). It is known \cite{Dra} that $\dim_G(X)\le\dim(X)$ for each abelian group $G$ and $\dim(X)=\dim_\IZ(X)$ for any finite-dimensional compact space $X$.
A group $G$ is called {\em 2-divisible} if for each $x\in G$ there is  $y\in G$ with $y^2=x$.

Theorem~\ref{main1} is completed by the following

\begin{theorem}\label{main2} If a family of rational functions $\F\subset\IR(x,y)$ contains a family of linear fractional transformations
$$\Big\{\frac{x-a}{y-b}:(a,b)\in D\Big\}.$$ for some dense subset $D$ of $\IR^2$, then the graphoid $\bar\Gamma(\F)$ of $\F$ has cohomological dimensions $\dim_\IZ(\bar\Gamma(\F))=\dim(\bar\Gamma(\F))=2$ and $\dim_G(\Gamma(\bar F))=1$ for any non-trivial 2-divisible abelian group $G$.
\end{theorem}

Theorems~\ref{main1} and \ref{main2} will be proved in Sections~\ref{s:t1} and \ref{s:t2}.
The main instrument in the proof of these theorems is Theorem~\ref{local} describing the local structure of the graphoid extension $\bar\F$ of a finite family of rational functions $\F\subset\IR(x,y)$.
 Section~\ref{s:prel} contains some notation and preliminary information, necessary for the proof of Theorem~\ref{local}.

\section{Preliminaries}\label{s:prel}

This section has preliminary character and contains notations and facts necessary for understanding the proof of Theorem~\ref{local}.

\subsection{Notation and Terminology} For two points $a,b\in\IR^2$ by $[a,b]=\{(1-t)a+tb:t\in[0,1]\}$ we shall denote the affine segment connecting $a$ and $b$ and by
$$\alpha_{a,b}:[0,1]\to[a,b],\;\;\alpha_{a,b}:t\mapsto(1-t)a+tb,$$the corresponding affine map.
Let also $]a,b[\;=[a,b]\setminus\{a,b\}$ be the open segment with the end-points $a,b$.
For a subset $A\subset\IR^2$ and a real number $t$ let $tA=\{ta:a\in A\}$ be a homothetic copy of $A$.
By $\bold 0=(0,0)$ we denote the {\em origin} of the plane $\IR^2$.

Two points $a,b$ of a subset $B\subset\IR^2\setminus\{\mathbf 0\}$ are called {\em neighbour points of $B$} if $a\ne b$ and $a,b$ are unique points of the set $B$ that lie in the convex cone $\{ua+vb:u,v\ge 0\}$. By $\mathcal N(B)$ (resp. $\mathcal N\{B\}$) we denote the family of ordered pairs $(a,b)\in B^2$ (resp. unordered pairs $\{a,b\}\subset B$) of neighbor points of $B$. This family will often occur in the proof of Theorem~\ref{local} below, so this is an important notion.

A subset $A$ of a metric space $(X,d)$ is called an {\em $\e$-net} in $X$ if for each point $x\in X$ there is a point $a\in A$ with $d(x,a)<\e$. For a point $z$ of a metric space $(X,d)$ and $\e>0$ let $B(z,\e)=\{x\in X:d(x,z)<\e\}$, $\bar B(z,\e)=\{x\in X:d(x,z)\le \e\}$, and $S(z,\e)=\{x\in X:d(x,z)=\e\}$ denote respectively the open $\e$-ball, closed $\e$-ball and $\e$-sphere centered at the point $z$.

A map $f:X\to Y$ between topological spaces $X,Y$ is {\em monotone} if $f^{-1}(y)$ is connected for each $y\in Y$.
It is easy to see that for a connected subspace $X\subset\IR$ a function $f: X\to\IR$ is monotone if and only if $f$ is either non-increasing or non-decreasing.

On the extended real line $\bar\IR=\IR\cup\{\infty\}$ we shall consider the metric $d$ inherited from the complex plane $\IC$ after the identification of $\bar\IR$ with the unit circle $\IT=\{z\in\IC:|z|=1\}$ with help of stereographic projection that maps $\IT\setminus\{i\}$ onto the real line $\IR$. In the metric $d$ the extended real line $\bar\IR$ has diameter 2. Observe that each (open or closed) ball in the metric space $(\bar\IR,d)$ is connected.

By an {\em arc} we understand a topological copy of the closed interval $[0,1]$. An arc $A$ in $\bar\IR^n$ is called {\em a monotone arc} if for each $i\in n$ the coordinate projection $\pr_i:A\to\bar\IR$ is a monotone map.

\subsection{Pusieux-analytic functions}
A function $\varphi:A\to\IR$ defined on a subset $A\subset \IR$ is called {\em analytic} if for every $a\in A$ there are $\e>0$ and real coefficients $(c_n)_{n\in\w}$ such that $\sum_{n=0}^\infty |c_n|\e^n<\infty$ and for every $x\in A$ with $|x-a|<\e$ we get $f(x)=\sum_{n=0}^\infty c_n(x-a)^n$.

Let $\e$ be a positive real number. A function $\varphi:[0,\e]\to\IR$ is called {\em Pusieux-analytic} if $\varphi|(0,\e]$ is analytic and there are $m\in\IN$, $\delta\in (0,\e)$ and an analytic function $\psi:[0,\sqrt[m]{\delta})\to\IR$ such that $\varphi(x)=\psi(\kern-3pt\sqrt[m]{x})$ for all $x\in[0,\delta)$.
The smallest number $m$ with this property is called the {\em Pusieux denominator} of $\varphi$. In a neighborhood of zero a Pusieux-analytic function $\varphi(x)$ develops into a series $\sum_{k=0}^\infty c_k x^{\frac{k}{m}}$ called the {\em Newton-Pusieux series} of $\varphi$, see \cite[8.3]{BriKno}. The interval $[0,\e]$ will be called the {\em domain} of the Pusieux analytic function $\varphi$ and will be denoted by $\dom(\varphi)$.

The Uniqueness Theorem for analytic functions (see e.g., \cite{Rud}) implies the following Uniqueness Theorem for Pusieux-analytic functions.

\begin{theorem} Two Pusieux-analytic functions $f,g:[0,\e]\to\IR$ are equal if and only if the set $\{x\in[0,\e]:f(x)=g(x)\}$ is infinite.
\end{theorem}

The Pusieux analycity can  be also introduced for functions defined on an interval $[-\e,0]$. Namely, we say that a function $\varphi:[-\e,0]\to\IR$ is {\em Pusieux-analytic} if the function $\psi:[0,\e]\to\IR$, $\psi:x\mapsto \varphi(-x)$, is Pusieux analytic.

Two Pusieux analytic function $\varphi$, $\varphi^*$ are called {\em conjugate} if they have the same Pusiuex denominator $m$ and for some analytic function $\psi:(-\delta,\delta)$ we get
$$\{(t^m,\psi(t)):|t|<\delta\}=\{(x,\varphi(x)):x\in\dom(\varphi)\cap(-\delta^m,\delta^m)\}\cup
\{(x,\varphi^*(x)):x\in\dom(\bar\varphi)\cap(-\delta^m,\delta^m)\}.$$

It can be shown that the Pusiuex denominator $m$ of two conjugate Pusieux analytic functions $\varphi,\varphi^*$ is odd if and only if $\dom(\varphi)\cap\dom(\varphi^*)=\{0\}$.

For example, the Pusieux analytic functions $\varphi_1:[-\e,0]\to\IR$, $\varphi_1:x\mapsto x^{\frac13}$, and $\varphi^*_1:[0,\e]\to\IR$, $\varphi^*_1:x\mapsto x^{\frac13}$, are conjugate and have Pusieux denominator 3.

The Pusieux analytic functions $\varphi_2:[0,\e]\to\IR$, $\varphi_2:x\mapsto x^{\frac32}$, and $\varphi_2^*:[0,\e]\to\IR$, $\varphi^*_2:x\mapsto -x^{\frac32}$, are conjugate and have Pusieux denominator 2.

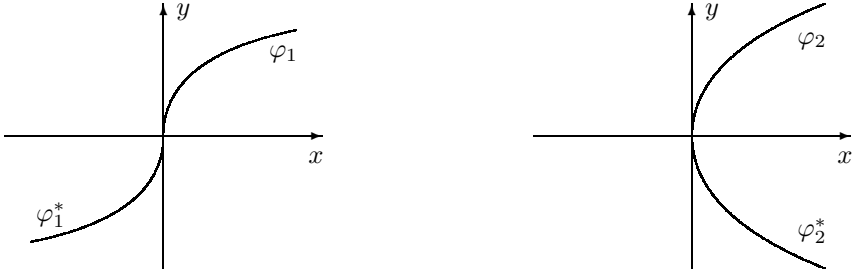
\begin{picture}(200,140)(-100,-60)
\put(-60,0){\vector(1,0){120}}
\put(0,-50){\vector(0,1){100}}
\put(55,-10){$x$}
\put(5,46){$y$}
\qbezier(0,0)(0,30)(50,40)
\put(40,30){$\varphi_1$}
\qbezier(0,0)(0,-30)(-50,-40)
\put(-48,-32){$\varphi^*_1$}

\put(140,0){\vector(1,0){120}}
\put(200,-50){\vector(0,1){100}}
\put(255,-10){$x$}
\put(205,46){$y$}
\qbezier(200,0)(200,30)(250,50)
\put(240,35){$\varphi_2$}
\put(240,-38){$\varphi^*_2$}
\qbezier(200,0)(200,-30)(250,-50)
\end{picture}

\begin{lemma}\label{l:conjug} If $\varphi,\varphi^*$ are two conjugate Pusieux analytic functions, then for any rational function $f\in\IR(x,y)$ the limits
$\lim_{x\to 0}f(x,\varphi(x))$ and $\lim_{x\to 0}f(x,\varphi^*(x))$ exist and are equal.
\end{lemma}

\begin{proof} The lemma is trivial if the rational function $f$ is constant. If $f$ is not constant we can write it as the fraction $f=\frac{p}{q}$ of two relatively prime polynomials $p$ and $q$. Observe that for each analytic function $\psi:[-\delta,\delta]\to\IR$ and any $m\in\IN$ the functions $p(t^m,\psi(t))$ and $q(t^m,\psi(t))$ are analytic and hence develop into Maclaurin series at a neighborhood of zero. This fact can be used to show that a (finite or infinite) limit
 $$\lim_{t\to 0}f(t^m,\psi(t))=\lim_{t\to 0}\frac{p(t^m,\psi(t))}{q(t^m,\psi(t))}$$
exits.

Now let $m$ be the Pusieux denominator of the conjugated Pusieux-analytic functions $\varphi$ and $\varphi^*$ and $\psi:[-\delta,\delta]\to\IR$ be an analytic function such that
$$\{(t^m,\psi(t)):|t|<\delta\}=\{(x,\varphi(x)):x\in\dom(\varphi)\cap(-\delta^m,\delta^m)\}
\cup\{(x,\varphi^*(x)):x\in\dom(\varphi^*)\cap(-\delta^m,\delta^m)\}.$$
It follows that
$$\lim_{x\to 0}f(x,\varphi(x))=\lim_{t\to 0}f(t^m,\psi(t))=\lim_{x\to0}{f(x,\varphi^*(x))}.$$
\end{proof}

\subsection{A local structure of a plane algebraic curve}

In this section we recall the known description of the local structure of an algebraic curve.

By an {\em algebraic curve} we understand a set of the form $$p^{-1}(0)=\{(x,y)\in\IR^2:p(x,y)=0\}$$ where $p\in\IR(x,y)$ is a non-zero polynomial of two variables with real coefficient. The polynomial $p$ in this definition can be also replaced by a non-zero rational function $r=\frac{p}{q}$ where $p$ and $q$ are two relatively prime polynomials. In this case the symmetric difference of the algebraic curves $r^{-1}(0)=\{(x,y)\in \dom(r):r(x,y)=0\}$ and $p^{-1}(0)$ lies in the intersection $p^{-1}(0)\cap q^{-1}(0)$, which is finite according to the classical B\'ezout Theorem or \cite[6.1]{BriKno} or  \cite[5.7]{Kunz}.

We are going to describe the structure of an algebraic curve $A\subset\IR^2$ at a neighborhood of zero $\mathbf 0=(0,0)$.

By $K=({-}1,1)^2$ we shall denote the open square with side $2$ centered at the origin $\mathbf 0$ of the plane and by $\bar K$ and  $K_\partial$ its closure and its boundary in the plane $\IR^2$.
Let $\bar K_\circ=\bar K\setminus\{\mathbf 0\}$ be the square $K$ with removed centrum and $\{\pm1\}^2=\{-1,1\}^2$ be the set of the vertices of the square.

Next, decompose the square $\bar K$ into four triangles:
\begin{itemize}
\item $\bar K_N=\{(x,y)\in \bar K:|x|\le y\}$,
\item $\bar K_{W}=\{(x,y)\in \bar K:|y|\le-x\}$,
\item $\bar K_{S}=\{(x,y)\in \bar K:|x|\le-y\}$,
\item $\bar K_{E}=\{(x,y)\in \bar K:|y|\le x\}$,
\end{itemize}
whose indices $N$, $W$, $S$, $E$ correspond to the directions: North, West, South and East.

A subset $C\subset \IR^2$ is called an {\em east $\e$-elementary curve} if
$C\subset \e\bar K_E$ and $C=\{(x,\varphi(x)):x\in(0,\e]\}$ for a (unique) Pusieux-analytic function $\varphi:[0,\e]\to\IR$. The Pusieux denominator $m$ of $\varphi$ will be called the {\em Pusieux denominator} of $C$.

An east $\e$-elementary curve $C$ is drawn on the following picture:

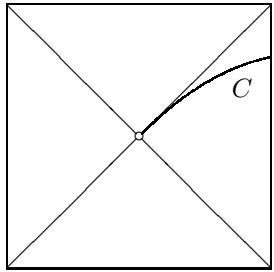
\begin{picture}(100,120)(-160,-10)
\put(0,0){\line(1,0){100}}
\put(100,0){\line(0,1){100}}
\put(100,100){\line(-1,0){100}}
\put(0,100){\line(0,-1){100}}

\put(0,0){\line(1,1){49}}
\put(100,0){\line(-1,1){49}}
\put(0,100){\line(1,-1){49}}
\put(100,100){\line(-1,-1){49}}

\put(50,50){\circle{3}}
\qbezier(51,51)(75,75)(100,80)
\put(85,65){$C$}
\end{picture}

The definitions of north, west, and south $\e$-elementary curves can be obtained by ``rotating'' the definition of an east $\e$-elementary curve.

Namely, let $R_{\frac\pi2}:(x,y)\mapsto (y,-x)$ be the clockwise rotation of the plane on the angle $\frac{\pi}2$.
Then $R_{\pi}=R_{\frac\pi2}\circ R_{\frac\pi2}$ and $R_{\frac{3\pi}2}=R_\pi\circ R_{\frac\pi2}$ are the clockwise rotations of the plane by the angles $\pi$ and $\frac{3\pi}2$, respectively.

A subset $C\subset\IR^2$ is called {\em north} (resp. {\em west}, {\em south}) {\em $\e$-elementary curve} if $R_{\frac\pi2}(C)$ (resp. $R_{\pi}(C)$, $R_{\frac{3\pi}2}(C)$) is an east $\e$-elementary curve. A subset $C\subset\IR^2$ will be called an {\em $\e$-elementary curve} if $C$ is an east, north, west or south $\e$-elementary curve.

We shall exploit the following fundamental fact describing the local structure of a plane algebraic curve, see \cite[\S 8.3]{BriKno} or \cite[\S16]{Kunz}.

\begin{theorem}\label{Pus} For any algebraic curve $A\subset\IR^2$ there is $\e>0$ such that the intersection $A\cap \e\bar K_\circ$ has finitely many connected components and each of them is an $\e$-elementary curve.
\end{theorem}

For an algebraic curve $A\subset\IR^2$ the number $\e>0$ satisfying the condition of Theorem~\ref{Pus} will be called {\em $A$-small}.

For an $A$-small number $\e$ each connected components of $A\cap\e\bar K$ is an $\e$-elementary curve called an {\em $\e$-branch} of $A$. Each $\e$-branch $C$ of $A$ has a conjugated $\e$-branch $C^*$ of $A$ defined as follows.

Assume first that the $\e$-branch $C$ is an east $\e$-elementary curve. Then $C=\{(x,\varphi(x)):x\in(0,\e]\}$ for some Pusieux-analytic function $\varphi:[0,\e]\to[-\e,\e]$ with Pusieux denominator $m$. For the function $\varphi$ there exist a positive $\delta\le\sqrt[m]{\e}$ and an analytic function $\psi:[-\delta,\delta]\to[-\e,\e]$ such that $\varphi(x)=\psi(\kern-3pt\sqrt[m]{x})$ for all $x\in[0,\delta^m]$.

If $m$ is odd, then the formula $\varphi^*(x)=\psi(\kern-3pt{\sqrt[m]{x}}\,)$ determines a Pusieux-analytic function $\varphi^*:[-\delta^m,0]\to\IR$, which is conjugate to $\varphi$. If $m$ is even, then the conjugate function $\varphi^*:[0,\delta^m]\to\IR$ is defined by the formula $\varphi^*(x)=\psi({-}\kern-2pt\sqrt[m]{x})$.

We claim that the graph $\{(x,\varphi^*(x)):x\in\dom(\varphi^*)\}$ lies in some $\e$-branch $C^*$ of the algebraic curve $A$. Find a polynomial $p\in\IR(x,y)$ such that $A=p^{-1}(0)$. Taking into account that $C\subset A$, we conclude that $p(x,\varphi(x))=0$ for all $x\in[0,\e]$ and hence $p(t^m,\psi(t))=0$ for all $t\in [0,\delta]$. Taking into account that the formula $f(t)=p(t^m,\psi(t))$ determines an analytic function $f:[-\delta,\delta]\to\IR$, which is zero on $[0,\delta]$, we conclude that $f\equiv 0$.

If $m$ is odd, then for every $x\in[-\delta^m,0]=\dom(\varphi^*)$ and $t=\sqrt[m]{x}$, we get $p(x,\varphi^*(x))=p(t^m,\psi(t))=0$.
If $m$ is even, then for every $x\in[0,\delta^m]=\dom(\varphi^*)$ and $t=-\sqrt[m]{x}$, we get $p(x,\varphi^*(x))=p(t^m,\psi(t))=0$.
Therefore the graph $\{(x,\varphi^*(x)):x\in\dom(\varphi^*)\setminus\{0\}\}$ lies in the algebraic curve $A$ and being a connected subset of $A\cap\e\bar K_\circ$ lies in a unique branch $C^*$, which is called the {\em conjugate $\e$-branch} of the $\e$-branch $C$. Observe that the conjugate branch $C^*$ is an east $\e$-elementary curve if $m$ is even and west if $m$ is odd.

By analogy we can define conjugate branches of north, west and south $\e$-branches of the algebraic curve $A$. Since the conjugate Pusieux analytic curves are not equal, the conjugated $\e$-branches of $A$ are disjoint.

So, the intersection $A\cap \e\bar K_\circ$ decomposes into the union of conjugated branches and hence contains an even number of connected components. This is a crucial observation which will be used in the proof of the inequality $\dim \bar\Gamma(\F)\ge 2$ in Theorem~\ref{main1}.

\subsection{Degree of maps between circles}

In this section we recall some basic information about the degree of maps between circles. Since the degree has topological nature, instead of the circle we can consider the boundary $K_\partial$ of the square $K=(-1,1)^2$ in the plane $\IR^2$.

We assume that the reader knows Elements of Singular Homology Theory with coefficients in an abelian group $G$ at the level of Chapter 2 of Hatcher's monograph \cite{Hat}. In particular, we assume that the reader knows the definition of the first homology group $H_1(X;G)$ of a topological space $X$ and also that each continuous map $f:X\to Y$ induces a homomorphism $f_*:H_1(X;G)\to H_1(Y;G)$ of the corresponding homology groups. It is well-known that the first homology group $H_1(K_\partial;G)$ of the (topological) circle $K_\partial$ is isomorphic to the group $G$, see \cite[p.153]{Hat}. In particular, for the infinite cyclic group $G=\IZ$ the first homology group $H_1(K_\partial;\IZ)$ is isomorphic to $\IZ$.

Observe that each homomorphism $h:\IZ\to\IZ$ is of the form $h(x)=d\cdot x$ for some integer number $d$ called the {\em degree} of the homomorphism $h$.
By the {\em degree}  of a continuous map $f:K_\partial \to K_\partial$ we understand the degree of the induced homomorphism $f_*:H_1(K_\partial;\IZ)\to H_1(K_\partial;\IZ)$.

For the coefficient group $\IZ_2=\IZ/2\IZ$ the situation simplifies. There are only two homomorphisms from $\IZ_2$ to $\IZ_2$: identity and trivial (or annulating). So, we call a map $f:X\to Y$ between topological circles {\em $\IZ_2$-trivial} if the induced homomorphism $f_*:H_1(X;\IZ_2)\to H_1(Y;\IZ_2)$ is trivial.
It is easy to see that a map $f:K_\partial\to K_\partial$ is $\IZ_2$-trivial if and only if it has even degree.

We shall need the following fact whose proof can be found in \cite[\S2.2]{Hat}.

\begin{lemma}\label{l:degree} A map $f:K_\partial\to K_\partial$ has even degree and is $\IZ_2$-trivial if for some point $y\in K_\partial$ the preimage $f^{-1}(y)$ has finite even cardinality and each point $x\in f^{-1}(y)$ has a neighborhood $U_x\subset K_\partial$ such that the restriction $f|U_x:U_x\to K_\partial$ is monotone.
\end{lemma}

Another result that will be used in the proof of Theorem~\ref{main1} is the addition formula for degrees, which in $\IZ_2$-case looks as follows:

\begin{lemma}\label{l:2.5} Let $Z$ be a finite subset of the open square $K=(-1,1)^2$ in the plane $\IR^2$ endowed with the max-norm, and $\e>0$ be a number such that $K_\partial\cap \bar B(Z,\e)=\emptyset$ and $\bar B(z,\e)\cap\bar B(z',\e)=\emptyset$ for distinct points $z,z'\in Z$. Let $f:\bar K\setminus B(Z,\e)\to K_\partial$ be a continuous map. If for every $z\in Z$ the restriction $f|S(z,\e):S(z,\e)\to K_\partial$ is a $\IZ_2$-trivial map, then the restriction $f|K_\partial:K_\partial\to K_\partial$ also is $\IZ_2$-trivial.
\end{lemma}

\begin{proof} Let $\sigma:[0,1]\to K_\partial$ be a continuous map such that $\sigma(0)=\sigma(1)$ and $\sigma|[0,1):[0,1)\to K_\partial$ is bijective. By \cite[2.23]{Hat}, its homology class $[\sigma]$ is a generator of the homology group $H_1(K_\partial;\IZ_2)$, which is isomorphic to $\IZ_2$.

Let $X=\bar K\setminus B(Z,\e)$ and $f_*:H_1(X;\IZ_2)\to H_1(K_\partial;\IZ_2)$ denote the homomorphism between the first homology groups, induced by the map $f:X\to K_\partial$.
Let $i:K_\partial \to X$ denote the identity embedding.

For every $z\in Z$ consider the singular simplex $\sigma_z:[0,1]\to S(z,\e)$, $\sigma_z:t\mapsto z+\e\sigma(t)$, whose homology class is a generator of the homology group $H_1(S(x,\e);\IZ_2)$ which is isomorphic to $\IZ_2$. Since the composition  $f|S(z,\e):S(z,\e)\to K_\partial$ is $\IZ_2$-trivial, $f_*([\sigma_z])=0$.

It is easy to show that the 1-cycle $\sigma-\sum_{z\in Z}\sigma_z$ is equal to the boundary of some singular 2-chain in $X$. Consequently, $f_*([\sigma])=\sum_{z\in Z}f_*([\sigma_z])=0$ and $f_*\circ i_*=0$, which means that the map $f|K_\partial=f\circ i$ is $\IZ_2$-trivial.
\end{proof}

\section{Resolving the singularity of a rational vector-function}

In this section given a finite non-empty family $\F\subset\IR(x,y)$ thought as a rational vector-function, we study the local structure of its canonical multi-valued extension $\bar\F:\IR^2\setmap \bar\IR^\F$ at a neighborhood of an arbitrary point $(a,b)\in\IR^2$. We lose no generality assuming the point $(a,b)$ coincides with the  origin $\mathbf 0=(0,0)$ of the plane $\IR^2$.

The principal result of this section is the following structure theorem.

\begin{theorem}\label{local} Let $\F\subset\IR(x,y)$ be a non-empty finite family of rational functions and $\bar \F:\IR^2\setmap \bar\IR^\F$ be its graphoid extension. There is $\tilde\e>0$ such that for every  $\e\in(0,\tilde \e]$ there is a homeomorphism $h:\e\bar K\setminus\frac\e2\bar K\to\e\bar K_\circ$ such that
\begin{enumerate}
\item $\e\bar K_\circ\subset\dom(\F)$.
\item $h|\e K_\partial=\mathrm{id}$.
\item For every $f\in \F$ the composition $f\circ h:\e\bar K\setminus \frac\e2\bar K\to\bar\IR$ has a continuous extension $\bar f_h:\e\bar K\setminus\tfrac\e2K\to\bar\IR$.
\item The functions $\bar f_h$, $f\in\F$, compose a continuous extension $$\bar \F_h=(\bar f_h)_{f\in \F}:\e\bar K\setminus\tfrac\e2K\to\bar\IR^\F$$ of $\F\circ h$ such that $\bar \F_h(\frac\e2K_\partial)=\bar\F(\mathbf 0)$.
\item There is a finite subset $B_0$ of $\e K_\partial$ containing the set $\{-\e,\e\}^2$ of vertices of $\e K$ such that for any neighbor points $a,b$ of $\frac12B_0$ and every  $f\in \F$ the restriction $\bar f_h|[a,b]:[a,b]\to \bar\IR$ is monotone and the image $\bar f_h([a,b])$ lies in one of the segments $[0,1]$, $[-1,0]$, $[1,\infty]$, $[\infty,-1]$ composing the circle $\bar\IR$.
\item The set $\bar\F(\mathbf 0)$ is either a singleton or a finite union of monotone arcs in $\bar\IR^\F$.
\item For any continuous map $g:\bar\F(\mathbf 0)\to K_\partial$ the composition $g\circ\bar\F_h|\frac\e2K_\partial:\frac\e2K_\partial\to K_\partial$ is $\IZ_2$-trivial.
\end{enumerate}
\end{theorem}

\begin{proof} We lose no generality assuming that all functions $f\in\F$ are not constant. Observe that for each rational function $f=\frac{p}{q}\in\F\subset\IR(x,y)$ the set $\IR^2\setminus\dom(f)\subset p^{-1}(0)\cap q^{-1}(0)$ is finite according to the classical theorem of B\'ezout \cite[6.1]{BriKno}  (which says that for two relatively prime polynomials $p,q\in\IR(x,y)$ the algebraic curves $p^{-1}(0)$ and $q^{-1}(0)$ have finite intersection).
This implies that the set $\dom(\F)=\bigcap_{f\in\F}\dom(f)$ is cofinite in $\IR^2$ (i.e., has finite complement in $\IR^2$).

In the family $\F$ consider the subfamilies:
\begin{itemize}
\item $\F_x$ of rational functions $f\in\F$ with non-zero partial derivative $f_x=\frac{\partial f}{\partial x}$;
\item $\F_y$ of rational functions $f\in\F$ with non-zero partial derivative $f_y=\frac{\partial f}{\partial y}$.
\end{itemize}

Let $C_0=\{0,1,-1,\infty\}$, $X=\{(x,y)\in\IR^2:x^2=y^2\}$, and consider the algebraic curve
$$A_{0}=X\cup\bigcup_{f\in \F_{\kern-1pt x}}f_x^{-1}(0)\cup \bigcup_{f\in\F_{\kern-1pt y}}f_y^{-1}(0)\cup\bigcup_{f\in\F}f^{-1}(C_0).$$

Using Theorem~\ref{Pus}, choose an $A_{0}$-small number $\tilde \e\in (0,1)$ such that $\tilde \e\bar K_\circ\subset\dom(\F)$. For this number $\tilde \e$ the intersection $A_{0}\cap\tilde \e\bar K_\circ$ decomposes into even number of pairwise disjoint $\tilde \e$-elementary curves.
Since the set $\IR^2\setminus \dom(\F)$ is finite, we can assume that $\tilde \e$ is so small that $\tilde \e K_\circ\subset \dom(\F)$.
Now, given any real number $\e\in(0,\tilde\e]$ we shall construct a homeomorphism $h:\e\bar K\setminus\frac\e 2K\to\e K_\circ$ that satisfies the conditions (1)--(7) of Theorem~\ref{local}.

Let us recall that $d$ stands for the metric on the extended real line $\bar\IR$ identified with the unit circle in the complex plane via the stereographic projection. This metric induced the max-metric
$$d^\F\big((x_f),(y_f)\big)=\max_{f\in\F}d(x_f,y_f)$$on the $\F$-torus $\bar\IR^\F$.

Using Theorem~\ref{Pus}, by induction we can construct a sequence of algebraic curves $(A_n)_{n=1}^\infty$, a sequence of real numbers $(\e_n)_{n=1}^\infty$ and a sequence of finite subsets $(C_n)_{n\in\w}$ of $\bar\IR$ such that for every $n\in\IN$ the following conditions hold:
\begin{enumerate}
\item $0<\e_{n}<\min\{\e_{n-1},2^{-n}\}$;
\item the number $\e_n$ is $A_n$-small;
\item the set $C_{n+1}$ contains $C_n$ and is a finite $2^{-n}$-net in $(\bar \IR,d)$;
\item $A_{n+1}=A_n\cup\bigcup\limits_{f\in \F}f^{-1}(C_{n+1})$.
\end{enumerate}

Let $\e_0=\e$. Now we are ready to construct a homeomorphism $h:\e\bar K\setminus\frac\e2 \bar K\to\e\bar K_\circ$ required in Theorem~\ref{local}. This homeomorphism will be defined recursively with help of the algebraic curves $A_n$, $n\in\w$.

For every $n\in\w$ consider the finite set $B_n=A_n\cap\e_n K_\partial$ in the boundary of the square $\e_n K$. It follows from $X\subset A_0$ that the set $B_0$ contains the set $\{-\e_n,\e_n\}^2$ of vertices of the square $\e_n K$.
For each point $b\in B_n$ there is a unique $\e_n$-elementary branch $C_b$ of the algebraic curve $A_n$ such that $\{b\}=C_b\cap A_n$.
\smallskip

For every $n\in\w$ let $\delta_n=\frac\e2+\frac{\e_n}2$ and observe that $\lim_{n\to\infty}\delta_n=\frac\e2$.
\smallskip

Let $h_{-1}$ be the identity map of $\e K_\partial$. By induction, for every $n\in\w$ we shall define a subset $B_n'\in\delta_n K_\partial$ and a homeomorphism $h_n:\delta_n\bar K\setminus \delta_{n+1}K\to \e_n\bar K\setminus \e_{n+1}K$ such that:
\begin{enumerate}
\item[(5)] $h_n\big((\frac\e2+\frac t2)K_\partial\big)=t K_\partial$ for each $t\in[\e_{n+1},\e_n]$;
\item[(6)] $B_{n}'=h_{n-1}^{-1}(B_{n})\subset\delta_n K_\partial$;
\item[(7)] for any $b'\in B_n'$ we get $h_n([\frac{\delta_{n+1}}{\delta_n},1]b')=C_{b}\setminus \e_{n+1}K$ where $b=h_{n-1}(b')\in B_n$;
\item[(8)] for any neighbor points $a,b\in B_n'$ and any $t\in[\frac{\delta_{n+1}}{\delta_n},1]$ the map $h_n|[ta,tb]$ is affine, which means that $h_n((1-u)ta+utb)=(1-u)h_n(ta)+uh_n(tb)$ for all $u\in[0,1]$.
\end{enumerate}

The conditions (6)--(8) imply that for every $n\in\w$ we get $h_{n-1}|\delta_n K_\partial=h_n|\delta_n K_\partial$. So, we can define a homeomorphism $h:\e\bar K\setminus\frac\e2\bar K\to\e\bar K_\circ$ letting $h|\delta_n\bar K\setminus \delta_{n+1}K=h_n$ for all $n\in\w$. The properties (5)--(8) of the homeomorphisms $h_n$ imply that the homeomorphism $h$ has the following properties for every $n\in\w$:
\begin{enumerate}
\item[(9)] $h_n\big(\frac\e2+\frac t2)K_\partial)=t K_\partial$ for each $t\in(0,\e]$;
\item[(10)] $B_{n}'=h^{-1}(B_{n})$;
\item[(11)] for any $b'\in B_n'\subset\delta_n K_\partial$ we get $h((\frac{\e}{2\delta_n},1]b')=C_{b}$ where $b=h(b')\in B_n$;
\item[(12)] for any neighbor points $a,b$ of $B_n'$ and any $t\in[\frac{\delta_{n+1}}{\delta_n},1]$ the map $h|[ta,tb]$ is affine.
\end{enumerate}

Moreover the choice of the algebraic curve $A_0$ guarantees that for any neighbor point $a,b\in B_0=B_0'$, any $t\in(1/2,1]$, and any function $f\in \F$
\begin{itemize}
\item[(13)] the restriction $f\circ h|[ta,tb]$ is either constant or injective (this follows from $f_x^{-1}(0)\cup f_y^{-1}(0)\subset A_0$) and
\item[(14)] the image $f\circ h([ta,tb])$ lies in one of segments $[0,1]$, $[-1,0]$, $[1,\infty]$, $[\infty,-1]$ composing the projective line $\bar\IR$ (this follows from $f^{-1}(\{0,1,-1,\infty\})\subset A_0$).
\end{itemize}

Now we shall prove the statements (1)--(7) of Theorem~\ref{local}. In fact, the statements (1) and (2) follow from the choice of $\e=\e_0$ and the definition of $h|\e K_\partial=h_{-1}$. The other statements will be proved in a series of claims and lemmas.

In the following claim (that proves the statement (3) of Theorem~\ref{local}) on the plane $\IR^2$ we consider the metric
 $$\rho\big((x,y),(x',y')\big)=\max\{|x-x'|,|y-y'|\}$$generated by the norm $\|(x,y)\|=\max\{|x|,|y|\}$. In this metric the square $K$ is just the open unit ball centered at $\mathbf 0$.

\begin{claim}\label{cl:t3} For every $f\in\F$ the map $f\circ h:\e\bar K\setminus\frac\e2\bar K\to \bar\IR$ is uniformly continuous and hence admits a continuous extension $\bar f_h:\e\bar K\setminus\frac\e2 K\to \bar\IR$.
\end{claim}

\begin{proof} Given any $\eta >0$, we should find $\tau>0$ such that for any two points $x,x'\in \e\bar K\setminus\frac\e2\bar K$ with $\rho(x,x')<\tau$ we get $d(f\circ h(x),f\circ h(x'))<\eta$. Choose a natural number $m\in N$ such that $2^{-m+2}<\eta$. By the uniform continuity of the function  $f$ on the compact set $\e\bar K\setminus \delta_{m+1} K$, there exists a real number $\tau_1 >0$ such that for any points $(x, x')  \in \e\bar K\setminus \delta_{m+1} K$ with $\rho(x,x')<\tau_1$ we have $d(f\circ h(x), f\circ h(x'))<\eta$. Let $\tau_2=\delta_m-\delta_{m+1}$ be equal to the smallest distance between the squares $\delta_m K_{\partial}$ and $\delta_{m+1} K_{\partial}$.

Now let us consider the finite set $\frac{\e}{2\delta_m} B_m'\subset \frac {\e}2 K_{\partial}$ and put $$\tau_3=\min\big\{\rho(a',b'): a',b' \in \tfrac \e{2\delta_m} B_m', a'\neq b'\big\}.$$ We claim that the number $\tau=\min\{\tau_1,\tau_2,\tau_3\}$ has the required property.

The choice of $\tau$ implies that any two points $x,x'\in \e\bar K\setminus \frac\e2\bar K$ with $\rho(x,x')<\tau$ either both lie in $\e\bar K\setminus \delta_{m+1} K$ (and by the definition of $\tau_1$ this implies that $d(f\circ h(x), f\circ h(x'))<\eta$), or they both lie in the same trapezoid $T_{ab}$, bounded by the lines $\delta_m K_{\partial}$, $\frac {\e}2 K_{\partial}$, $[\frac{\e}{2\delta_m}, 1]a, [\frac{\e}{2\delta_m}, 1]b$, where $a,b$ are neighbor points in $B_m'$, or, at least, in such two adjacent trapezoids. The interior $T_{ab}\setminus\partial T_{ab}$ of the trapezoid $T_{ab}$ is a connected set whose image $h(T_{ab}\setminus \partial T_{ab})$ does not intersect the algebraic curve $A_m$ while the image $f\circ h(T_{a,b}\setminus \partial T_{ab})$ does not intersect the $2^{-m+1}$-net $C_m$ in $\bar \IR$. Consequently, $\diam f\circ h(T_{ab})=\diam f\circ h(T_{ab}\setminus \partial T_{ab})<2^{-m+1}$ and $d(f\circ h(x),f\circ h(x'))\le 2\cdot 2^{-m+1}<\eta$.
\end{proof}

The functions $\bar f_h$, $f\in\F$, compose a continuous function $\bar \F_h=(\bar f_h)_{f\in\F}:\e\bar K\setminus\frac\e2K\to\bar\IR^\F$ that extends the composition $\F\circ h:\e\bar K\setminus\frac\e2\bar K\to\bar\IR^\F$.
Let $\bar\F_\partial=\bar\F_h|\frac\e2K_\partial$ be the restriction of $\bar\F_h$ onto the boundary square $\frac\e2K_\partial$. Also let $\bar h:\e\bar K\setminus\frac\e2K\to\e\bar K$ be the continuous extension of the homeomorphism $h$ and observe that $\bar h^{-1}(\bold 0)=\frac\e2 K_\partial$.

The following claim completes the proof of the statement (4) of Theorem~\ref{local}.

\begin{claim}\label{cl:t4} $\bar \F(\mathbf 0)=\bar \F_\partial(\frac \e2K_\partial)$ is a Peano continuum.
\end{claim}

\begin{proof}
First, we are going to show that $\bar \F(\mathbf 0)=\bar \F_\partial(\frac \e2K_\partial)$.

Let $y\in \bar \F(\mathbf 0)$. This means that there exists a sequence $\{x_n\}_{n\in\w}\subset \e K_\circ$ such that $(\mathbf 0,y)=\lim\limits_{n\to\infty}(x_n,\F(x_n))$. By the compactness of $\e\bar K\setminus\frac\e2K$, the sequence $\{h^{-1}(x_n)\}_{n\in\w}\subset \e\bar K\setminus\frac\e2 \bar K$ contains a subsequence $\{h^{-1}(x_{n_k})\}_{k\in\w}$ that converges to some point $z\in \frac\e2K_\partial=\bar h^{-1}(\bold 0)$. The continuity of the map $\bar \F_h$ guarantees that $$y=\lim_{k\to\infty}\F(x_{n_k})=\lim_{k\to\infty}\bar \F_h(h^{-1}(x_{n_k}))=\bar\F_h(z)\in\bar\F_\partial(\tfrac\e2K_\partial).$$
The converse inclusion $\bar\F_\partial(\frac\e2K_\partial)\subset\bar\F(\bold 0)$ is obvious. The equality $\bar\F(\bold 0)=\bar \F_\partial(\frac \e2K_\partial)$ and the continuity of $\bar\F_\partial$ implies that $\bar\F(\bold 0)$ is a Peano continuum.
\end{proof}

Now we prove the statement (5) of Theorem~\ref{local}. We recall that $B_0'=B_0=A_0\cap\e K_\partial$. The conditions (13), (14) imply the following:

\begin{claim}\label{cl:t5} For every $f\in \F$ and neighbor points $a,b$ of the set $\frac12B_0$ the map $\bar f_h|[a,b]$ is monotone and its image lies in one of the segments: $[0,1]$, $[-1,0]$, $[1,\infty]$, $[\infty,-1]$ composing the projective line $\bar\IR$.
\end{claim}

Claim~\ref{cl:t5} implies:

\begin{claim}\label{cl:t6a} For any neighbor points $a,b$ of the set $\frac12B_0$, the map $\bar \F_\partial|[a,b]$ is monotone and $\bar \F_\partial([a,b])$ either is a singleton or a monotone arc in $\bar\IR^\F$.
\end{claim}

Claim~\ref{cl:t6a} implies that for any neighbor points $a,b$ of the set $\frac12B_0$ and any point $y\in\bar\F_\partial([a,b])$ the preimage $(\bar \F_\partial|[a,b])^{-1}(y)$ is either a singleton or an arc. Let
$$Y_{a,b}=\{y\in \bar\F_h([a,b]):|(\bar \F_\partial|[a,b])^{-1}(y)|>1\}.$$
Since $[a,b]$ does not contain uncountably many disjoint arcs, the set $Y_{a,b}$ is at most countable and so is the set
$$Y=\bigcup\{Y_{a,b}:(a,b)\in\mathcal N(\tfrac12B_0)\}.$$
The definition of the set $Y$ implies:

\begin{claim}\label{cl:t7a} For each $y\in\bar\IR^\F\setminus Y$ the preimage $\bar\F_\partial^{-1}(y)$ is finite.
\end{claim}

The last statement of Theorem~\ref{local} is proved in the following lemma, which is the most difficult part of the proof of Theorem~\ref{local}.

\begin{lemma}\label{l:even} For any map $g:\bar \F_\partial(\frac\e2K_\partial)\to\frac\e2 K_\partial$ the composition $g\circ\bar\F_\partial:\frac\e2K_\partial\to \frac\e2K_\partial$ is $\IZ_2$-trivial.
\end{lemma}

\begin{proof} Since homotopic maps have the same degrees, it suffices to show that the map $g\circ\bar\F_\partial$ is homotopic to some $\IZ_2$-trivial map $\tilde g:\frac\e2K_\partial\to \frac\e2K_\partial$.
The construction of such a map $\tilde g$ is rather long and require some preliminary work, in particular, introducing some notation.

We recall that $\rho$ stands for the $\max$-metric on the plane $\IR^2$, $d$ denotes the metric of the projective line $\bar\IR$. The latter metric induces the $\max$-metric $d^\F$ on the $\F$-torus $\bar\IR^\F$.

By the uniform continuity of the map $g$ there is $\delta>0$ such that for any points $x,y\in\bar\F(\bold 0)=\bar\F_\partial(\frac\e2K_\partial)$ with $d^\F(x,y)\le\delta$ we get $\rho(g(x),g(y))<\e$.

Find $m\in\IN$ such that $2^{-m+1}<\delta$ and consider the $\delta$-net $C_m$ in $\bar\IR$ (which appeared in the construction of the homeomorphism $h$).
The definition of the set $C_m$ guarantees that $C_0=\{0,1,-1,\infty\}\subset C_m$.
The set $C_m$ induces the disjoint cover $$\C=\big\{\{c\}:c\in C_m\}\cup\{]a,b[:(a,b)\in \mathcal N(C_m)\}$$ of the projective line $\bar\IR$.

It follows that the closure $\bar C$ of each set $C\in\C$ is a connected subset that lies in one of the intervals:  $[0,1]$, $[0,-1]$, $[1,\infty]$, $[\infty,-1]$. So, we can endow each segment $\bar C$ with the linear order inherited from the extended real line $[-\infty,\infty]$.

Now for each set $C\in\C$ consider the rational homeomorphism $\mu_C:\bar \IR\to\bar\IR$ defined by formula:
$$
\mu_C(x)=\begin{cases}
x&\mbox{if $C\subset [0,1]$};\\
x+1&\mbox{if $C\subset [-1,0)$};\\
1-x^{-1}&\mbox{if $C\subset(1,\infty]$};\\
-x^{-1}&\mbox{if $C\subset(\infty,-1)$}.
\end{cases}
$$
Observe that $\mu_C(\bar C)\subset[0,1]$ and the restriction $\mu_C|\bar C:\bar C\to[0,1]$ is strictly increasing (with respect to the linear order on $\bar C$ inherited from $[-\infty,+\infty]$).

The cover $\C$ induces the disjoint cover $$\Pi\C^\F=\Big\{\prod_{f\in\F}C_f:(C_f)_{f\in\F}\in\C^\F\Big\}$$ of the $\F$-torus $\bar\IR^\F$ by cubes of various dimensions. Since $C_m$ is a $\delta$-net, for each cube $C\in\Pi\C^\F$ its closure $\bar C$ has diameter $<\delta$ (with respect to the metric $d^\F$). Consequently, the image $g(\bar C\cap\bar\F(\bold 0))$ has $\rho$-diameter $\diam\, g(\bar C\cap \bar\F(\bold 0))<\e$ and hence lies in some topological arc $$I_C\subset \frac\e2 K_\partial.$$

For each cube  $C=\prod_{f\in\F}C_f\in\Pi\C^\F$ consider the embedding $$\mu_{C}:\prod_{f\in\F}\bar C_f\to[0,1]^\F,\;\;\mu_{C}:(x_f)_{f\in\F}\mapsto (\mu_{C_f}(x_f))_{f\in\F}.$$
\smallskip

And now the final portion of definitions and notations which should be digested before the start of the proof of Lemma~\ref{l:even}.

A pair $(a,b)$ of distinct points of $\frac\e2K_\partial$ is called {\em $\F$-admissible} if $[a,b]\subset \frac\e2K_\partial$ and for every $f\in\F$ there is a (unique) set $C^f_{a,b}\in\C$ such that $\bar f_h\big({]a,b[}\big)\subset C^f_{a,b}$ and the restriction $\bar f_h|[a,b]:[a,b]\to\bar C^f_{a,b}$ is monotone. It is clear that the product $$C_{a,b}=\prod_{f\in\F}C^f_{a,b}\subset\bar\IR^\F$$is an element of the cover $\Pi\C^\F$ of $\bar\IR^\F$.

For each pair $(a,b)$ of $\F$-admissible points of $\frac\e2K_\partial$ consider the sets
$$
\begin{aligned}
\F_{a,b}^{<}&=\{f\in\F:\bar f_h(a)<\bar f_h(b)\},\\
\F_{a,b}^{>}&=\{f\in\F:\bar f_h(a)>\bar f_h(b)\},\\
\F_{a,b}^{=}&=\{f\in\F:\bar f_h(a)=\bar f_h(b)\},\\
\F_{a,b}^{\ne}&=\F\setminus\F^{=}_{a,b}=\F_{a,b}^{<}\cup\F_{a,b}^{>}.
\end{aligned}
$$

Two $\F$-admissible ordered pairs $(a,b)$, $(a',b')$ of neighbor points of the set $\frac\e2K_\partial$ are called
{\em $\F$-coherent} if  $$\F_{a,b}^{<}=\F_{a',b'}^{<},\;\;\F_{a,b}^{=}=\F_{a',b'}^{=},\;\;\F_{a,b}^{>}=\F_{a',b'}^{>}, \mbox{ \ and \ }C_{a,b}=C_{a',b'}.$$

Two unordered pairs $\{a,b\}$, $\{a',b'\}$ of neighbor points of the set $B$ are called {\em $\F$-coherent} if the ordered pair $(a,b)$ is $\F$-coherent either to $(a',b')$ or to  $(b',a')$.

It is easy to check that the $\F$-coherence relation is an equivalence relation on the family of $\F$-admissible (un)ordered pairs of points of $\frac\e2K_\partial$.

\begin{claim}\label{cl:D} There is a finite subset $D\subset\frac\e2K_\partial$ such that:
\begin{enumerate}
\item $\frac12B_0\subset D$.
\item Any pair $(a,b)$ of neighbor points of $D$ is $\F$-admissible.
\item Two unordered pairs $\{a,b\},\{a',b'\}$ of neighbor points of $D$ are $\F$-coherent provided $\bar \F_\partial\big(]a,b[\big)\cap \bar \F_\partial\big(]a',b'[\big)\ne\emptyset$.
\end{enumerate}
\end{claim}

\begin{proof} For every neighbor points $a,b$ of the set $\frac12B_0$,  consider the disjoint cover $\DD_{a,b}=\{[a,b]\cap\bar\F_h^{-1}(C):C\in\Pi\C^\F\}$ of the affine interval $[a,b]$ by convex subsets of $[a,b]$.
The convexity of the sets of the cover $\DD_{a,b}$ follows from the monotonicity of the maps $\bar f_h|[a,b]$, $f\in\F$.
 For every set $D\in\DD_{a,b}$ by $\partial D$ we denote the boundary of $D$ in $\frac\e2K_\partial$. Since $D$ is convex,
$|\partial D|\le 2$. Let $$D_0=\tfrac12B_0\cup\bigcup\big\{\partial D:D\in\DD_{a,b},\;(a,b)\in\mathcal N(\tfrac12B_0)\big\}.$$
It is easy to see that any two neighbor points of the set $D_0$ are $\F$-admissible.

Now let us consider an increasing sequence $(D_n)_{n\in\w}$ of finite subsets of $\frac\e2K_\partial$ defined by the following recursive procedure. Assume that for some $n\in\w$ a finite subset $D_n$ (containing the set $D_0$) has been constructed. For any two unordered pairs $\{a,a'\}$ and $\{b,b'\}$ of neighbor points of the set $D_n$ consider the convex hull $D_{a,a'}^{b,b'}$ of the set
$(\bar \F_\partial|[a,a'])^{-1}(\bar \F_\partial([b,b']))\subset[a,a']$ in the affine segment $[a,a']$ and its boundary $\partial D_{a,a'}^{b,b'}$ in $\frac\e2K_\partial$, which consists of at most two points.

\begin{claim}\label{cl:D1} $\bar \F_\partial(\partial D_{a,a'}^{b,b'})=\bar\F_\partial(\partial D_{b,b'}^{a,a'})$. If the intersection $\bar\F_{\partial}([a,a'])\cap\bar\F_\partial([b,b'])$ contains more than one point, then doubletons $\partial D_{a,a'}^{b,b'}$ and $\partial D_{b,b'}^{a,a'}$ are $\F$-coherent.
\end{claim}

\begin{proof} The claim is trivial if the intersection $Z=\bar\F_{\partial}([a,a'])\cap\bar\F_\partial([b,b'])$ contains at most one point.
So, assume that this intersection contains more than one point.
It follows that $Z\subset C_{a,a'}=C_{b,b'}$ and hence for each $f\in \F$ the sets $C_{a,a'}^f$ and $C_{b,b'}^f$ coincide and carry the same linear order.

Choose two points $y=(y_f)_{f\in \F}$ and $y'=(y_f')_{f\in\F}$ in $Z\subset\bar\IR^\F$ for which the set $\F_{y,y'}^{\ne}=\{f\in\F:y_f\ne y_f'\}$ has maximal possible cardinality. It is clear that $\F^{\ne}_{y,y'}=\F^{<}_{y,y'}\cup \F^{>}_{y,y'}$, where
$\F_{y,y'}^{<}=\{f\in\F:y_f<y_f'\}$ and $\F_{y,y'}^{>}=\{f\in\F:y_f>y_f'\}$.

Choose two points $x_a,x_a'\in[a,a']$ such that $y=\bar\F_\partial(x_a)$ and $y'=\bar\F_\partial(x_a')$. Exchanging the points $a,a'$ by their places, if necessary, we can assume that the intervals $[a,x_a]$ and $[x_a',a']$ have empty intersection. Choose unique points $z_a\in [a,x_a]$ and $z'_a\in[x_a',a']$ such that $\{z_a,z_a'\}=\partial D_{a,a'}^{b,b'}$. Since $[x_a,x_a']\subset [z_a,z_a']$, the monotonicity of the functions $\bar f_h|[a,a']$ implies that
$\F_{y,y'}^{\ne}\subset \F_{z_a,z_a'}^{\ne}$ and hence $\F_{y,y'}^{\ne}=\F_{z_a,z_a'}^{\ne}$ by the maximality of $\F_{y,y'}^{\ne}$. This fact, combined with the monotonicity of the functions $\bar f_h|[a,a']$ and the choice of the order of the points $a,a'$ implies that $\F^{<}_{z_a,z_a'}=\F^<_{y,y'}$ and $\F^{>}_{z_a,z_a'}=\F^{>}_{y,y'}$. Let $y_a=\bar \F_\partial(z_a)$ and $y_a'=\bar\F_\partial(z'_a)$.

Now do the same for the segment $[b,b']$: choose two points $x_b,x_b'\in[b,b']$ such that $y=\bar \F_\partial(x_b)$ and $y'\in \bar \F_\partial(x_b')$. Replacing the points $b,b'$ by their places, if necessary, we can assume that the intervals $[b,x_b]$ and $[x_b',b']$ have no common points. Choose unique points $z_b\in [b,x_b]$ and $z'_b\in[x_b',b']$ such that $\{z_b,z_b'\}=\partial D_{b,b'}^{a,a'}$ and let $y_b=\bar\F_\partial (z_b)$ and $y'_b=\bar\F_\partial(z_b')$.
It follows that $\F^{>}_{z_b,z_b'}=\F^{>}_{y,y'}$ and $\F^{<}_{z_b,z_b'}=\F^{<}_{y,y'}$.

We claim that $y_a=y_b$ and $y_a'=y_b'$.
Assume first that $y_a\ne y_b$. Find a point $u_a\in[z_a,z_a']$ such that $\bar\F_\partial(u_a)=y_b$ and a points $u_b\in[z_b,z_b']$ such that $\bar\F_\partial(u_b)=y_a$. Since $y_a\ne y_b$, there is a function $f\in \F$ such that $\pr_f(y_a)\ne \pr_f(y_b)$ where $\pr_f:\bar\IR^\F\to\bar\IR$ denotes the projection onto the $f$-th coordinate. On the set $\pr_f(Z)$ consider a linear order inherited from the set $C_{a,a'}^f=C_{b,b'}^f$. We lose no generality assuming that $\pr_f(y_a)<\pr_f(y_b)$. Then by the monotonicity of the function $\bar f_h|[a,a']$, we get
$$\bar f_h(z_a)=\pr_f(y_a)<\pr_f(y_b)=\bar f_h(u_a)\le\bar f_h(z_a')$$ and hence $f\in \F^{<}_{z_a,z_a'}=\F^{<}_{y,y'}$. On the other hand, the monotonicity of the function $\bar f_h|[b,b']$ implies
$$\bar f_h(z_b)=\pr_f(y_b)>\pr_f(y_a)=\bar f_h(u_b)\ge\bar f_h(z_b'),$$ and  $f\in\F^{>}_{z_b,z_b'}=\F^{>}_{y,y'}$. This is a desired contradiction that proves the equality $y_a=y_b$. By analogy we can prove the equality $y_a'=y_b'$.

Now we see that the equality
$$\bar\F_\partial(\partial D_{a,a'}^{b,b'})=\{y_a,y_a'\}=\{y_b,y_b'\}=\bar\F_\partial(\partial D_{b,b'}^{a,a'})$$implies that the doubletons $\partial D_{a,a'}^{b,b'}$ and $\partial D_{b,b'}^{a,a'}$ are $\F$-coherent.
\end{proof}

Define the set $D_{n+1}$ as the union
$$D_{n+1}=D_n\cup\bigcup\{\partial D_{a,a'}^{b,b'}: (a,a'),(b,b')\in\mathcal N(D_n)\}.$$

For every $n\in\w$ consider the function $p_n:\mathcal N(D_{n+1})\to\mathcal N(D_n)$ assigning to each ordered pair $(a,a')$ of neighbor points of the set $D_{n+1}$ a unique ordered pair  $(b,b')$ of neighbor points of $D_n$ such that $[a,a']\subset[b,b']$ and $[a,b]\cap[a',b']=\emptyset$. For $n\le m$ consider the composition $$p_n^m=p_n\circ\dots\circ p_{m-1}:\mathcal N(D_m)\to\mathcal N(D_n).$$

\begin{claim}\label{cl:D2} There is $n\in\w$ such that for any $m\ge n$ any pair $(a',b')\in\mathcal N(D_{m})$ is $\F$-coherent to the pair $(a,b)=p_{m-1}(a',b')$.
\end{claim}

\begin{proof} The proof of this claim relies on the K\"onig Lemma \cite[14.2]{JW2}, which  says that a tree $T$ is finite provided each element of $T$ has finite degree and each branch of $T$ is finite. Let us recall that a {\em tree} is a partially ordered set (poset) $(T,\le)$ with the smallest element such that for each $t\in T$, the set $\{s \in T : s \le t\}$ is well-ordered by the relation $\le$. For each $t\in T$, the order type of $\{s\in T : s \le t\}$ is called the {\em height} of $t$. The height of $T$ itself is the least ordinal greater than the height of each element of $T$. The {\em degree} of an element $t\in T$  is the number of immediate successors of $t$ in $T$. The root of a tree $T$ is the unique element of height 0. A {\em branch of a tree} $T$ is a maximal linearly ordered subset of $T$.

Now consider the tree $T=\{\emptyset\}\cup\bigcup_{n\in\w}\mathcal N(D_n)$. The partial order on $T$ is defined as follows. Given two vertices $(a,a')\in \mathcal N(D_n)$ and $(b,b')\in\mathcal N(D_m)$ of $T$, we write $(a,a')\le (b,b')$ if $n\le m$ and  $(a,a')=p^m_n(b,b')$. The set $\emptyset$ is the root of $T$ and is smaller that any other non-empty element of $T$. It is clear that each vertex of the tree $T$ has finite degree.

The monotonicity of the maps $\bar f_h|[a,a']$ for $(a,a')\in T$ implies the following fact:

\begin{claim}\label{cl:new} For any two vertices $(a,a')\le (b,b')$ of the tree $T$ we get $\F^{=}_{a,a'}\subset \F^{=}_{b,b'}$. Moreover, the pairs $(a,a')$ and $(b,b')$ are $\F$-coherent if and only if
$\F^{=}_{a,a'}=\F^{=}_{b,b'}$.
\end{claim}

Now consider the subtree $T'\subset T$ consisting of the root of $T$ and all pairs $(a,a')\in \mathcal N(D_n)\subset T$ that are not $\F$-coherent to some pair $(b,b')\in \mathcal N(D_{n+1})\subset T$ with $p_n(b,b')=(a,a')$. Claim~\ref{cl:new} implies that each branch of the tree $T'$ has finite length $\le |\F|+1$. By K\"onig Lemma, the subtree $T'$ is finite. Consequently, there is $n\in\IN$ such that $T'\cap \mathcal N(D_m)=\emptyset$ for all $m\ge n-1$. This implies that for every $m\ge n$, each pair $(a,a')\in\mathcal N(D_{m})$ is $\F$-coherent to the pair $(b,b')=p_{m-1}(a,a')$. This completes the proof of Claim~\ref{cl:D2}.
\end{proof}

Let $D=D_{n+1}$ where the number $n$ is taken from Claim~\ref{cl:D2}. It is clear that the set $D$ satisfies the conditions (1) and (2) of Claim~\ref{cl:D}. The condition (3) is verified in the following claim.

\begin{claim}\label{cl:D3} Two unordered pairs $\{a,a'\},\{b,b'\}\in\mathcal N\{D\}$ of neighbor points of the set $D=D_{n+1}$ are $\F$-coherent if $\bar\F_h\big(]a,a'[\big)\cap\bar\F_h\big(]b,b'[\big)\ne\emptyset$.
\end{claim}

\begin{proof} We shall consider two cases (and several subcases).

1. The intersection $\bar\F_h\big(]a,a'[\big)\cap\bar\F_h\big(]b,b'[\big)$ contains more than one point. By Claim~\ref{cl:D1}, the doubletons $\partial D_{a,a'}^{b,b'}$ and $\partial D_{b,b'}^{a,a'}$ are
$\F$-coherent. Take a pair of neighbor points $(a_{n+2},a_{n+2}')\in\mathcal N(D_{n+2})$  such that $]a_{n+2},a'_{n+2}[\;\subset D_{a,a'}^{b,b'}$ and $p_{n+1}(a_{n+2},a_{n+2}')=(a,a')$.
 The choice of the number $n$ guarantees that the pairs $(a_{n+2},a_{n+2}')$ and $(a,a')$ are $\F$-coherent. Taking into account that $[a_{n+2},a_{n+2}']\subset \conv(\partial D_{a,a'}^{b,b'})\subset[a,a']$, we conclude that the pair $\{a,a'\}$ is $\F$-coherent to the doubleton $\partial D_{a,a'}^{b,b'}$. By analogy we can prove that the pair $\{b,b'\}$ is $\F$-coherent to the doubleton $\partial D_{b,b'}^{a,a'}$. Now we see that the $\F$-coherence of the doubletons $\partial D_{a,a'}^{b,b'}$ and $\partial D_{b,b'}^{a,a'}$ implies the $\F$-coherence of the pairs $\{a,a'\}$ and $\{b,b'\}$.

2. The intersection $\bar\F_h\big(]a,a'[\big)\cap\bar\F_h\big(]b,b'[\big)$ is a singleton containing a unique point $y$. If both sets $\partial D_{a,a'}^{b,b'}$ and $\partial D_{b,b'}^{a,a'}$ are doubletons, then we can use the equality $\bar \F_\partial (\partial D_{a,a'}^{b,b'})=\{y\}=\bar \F_\partial(\partial D_{b,b'}^{a,a'})$, which implies that the doubletons  $\partial D_{a,a'}^{b,b'}$ and $\partial D_{b,b'}^{a,a'}$ are $\F$-coherent and proceed as in the preceding case.

2a. Now assume that $\partial D_{a,a'}^{b,b'}$ is a singleton. Let $(a_n,a'_n)=p_{n+1}(a,a')$ and $(b_n,b_n')=p_{n+1}(b,b')$. The choice of the number $n$ guarantees that the pair $(a_n,a_n')$ is $\F$-coherent to $(a,a')$ and $(b_n,b_n')$ is $\F$-coherent to $(b,b')$. It follows that the intersection $$\bar\F_h\big(]a_n,a_n'[\big)\cap\bar\F_h\big(]b_n,b_n'[\big)\supset
\bar\F_h\big(]a,a'[\big)\cap\bar\F_h\big(]b,b'[\big)=\{y\}$$is not empty.
If this intersection is a singleton, then the convex set $D_{a_n,a_n'}^{b_n,b_n'}$ also is a singleton (in the opposite case, the set $D_{a,a'}^{b,b'}=D_{a_n,a_n'}^{b_n,b_n'}\cap\; ]a,a'[$ cannot be a singleton). In this case the singleton $\partial D_{a_n,a_n'}^{b_n,b_n'}=\partial D_{a,a'}^{b,b'}$ belongs to the set $D=D_{n+1}$ and is disjoint with the open interval $]a,a'[$, which contradicts $y\in\bar\F_\partial\big(]a,a'[\big)$.
This proves that the intersection $\bar\F_h\big(]a_n,a_n'[\big)\cap\bar\F_h\big(]b_n,b_n'[\big)$ is not a singleton. Proceeding as in the case 1, we can show that the pairs $\{a_n,a_n'\}$ and $\{b_n,b_n'\}$ are $\F$-coherent and so are the pairs $\{a,a'\}$ and $\{b,b'\}$ (which are $\F$-coherent to the pairs $\{a_n,a_n'\}$ and $\{b_n,b_n'\}$, respectively).

2b. In case $\partial D_{b,b'}^{a,a'}$ is a singleton, we can proceed by analogy with the case 2a.
\end{proof}
\end{proof}

Now we are ready to prove that the composition $g\circ\bar\F_\partial$ is homotopic to some a map $\tilde g:\frac\e2K_\partial\to\frac\e2K_\partial$ of even degree. It suffices to define $\tilde g$ on each segment $[a,b]$ connecting two neighbor points of the set $D$.

We recall that by $\mathcal N\{D\}$ we denote the family of unordered pairs of neighbor points of the set $D$. The family $\mathcal N\{D\}$ decomposes into pairwise disjoint equivalence classes consisting of $\F$-coherent pairs. Denote by  $\lrN\{D\}$ the family of these equivalence classes.
For each equivalence class $E\in\lrN\{D\}$ let
$$\bar E=\bigcup\big\{[a,b]:\{a,b\}\in E\big\}\mbox{ and }\partial E=\bigcup\big\{\{a,b\}:\{a,b\}\in E\big\}.$$
It is clear that $\frac\e2K_\partial=\bigcup\{\bar E:E\in\lrN\{D\}\}.$

For each equivalence class $E\in\lrN\{D\}$ we are going to construct a specific map $\tilde g_E:\bar E\to\frac\e2K_\partial$ such that $\tilde g_E|\partial E=g\circ\bar\F_\partial|\partial E$ and
$\tilde g_E$ is homotopic to $g\circ\bar \F_\partial|\bar E$. This map $\tilde g_E$ will have a specific algebraic structure which will help us to evaluate the degree of the unified map $\tilde g=\bigcup\big\{\tilde g_E:E\in\lrN\{D\}\big\}$.

So, fix an equivalence class $E\in\lrN\{D\}$. Since any two unordered pairs from $E$ are $\F$-coherent, we can choose a function $\gamma:E\to D^2$ assigning to each unordered pair $\{a,b\}\in E$ one of ordered pairs $(a,b)$ or $(b,a)$ so that for  any unordered pairs $\{a,b\},\{a',b'\}\in\lrN\{D\}$ the ordered pairs $\gamma(\{a,b\})$ and $\gamma(\{a',b'\})$ are $\F$-coherent. Let $\vec E=\gamma(E)\subset \mathcal N(D)$ and $\rN(D)=\{\vec E:E\in\lrN\{D\}\}$.
The $\F$-coherence of any two pairs $(a,b),(a',b')\in\vec E$ implies that $\F_{a,b}^{<}=\F_{a',b'}^{<}$, $\F_{a,b}^{=}=\F_{a',b'}^{=}$, $\F_{a,b}^{>}=\F_{a',b'}^{>}$, and $C_{a,b}=C_{a',b'}$. So, we can put
$$
\begin{aligned}
&\F_{E}^{<}=\F_{a,b}^{<},\;\F_E^{=}=\F_{a,b}^{=},\;\F_E^{>}=\F_{a,b}^{>},\\
&C_E^f=C_{a,b}^f \mbox{ \ for $f\in\F$}, C_E=C_{a,b}=\prod_{f\in\F}C_{a,b}^f,\\
&I_E=I_{C_E}\;\mbox{ and }\mu_E=\mu_{C_E}:\bar C_E\to[0,1]^\F
\end{aligned}
$$where $(a,b)\in\vec E$ is any pair. We recall that $I_{C_E}$ is an arc in $\frac\e2 K_\partial$ that contains the set $g(\bar C_E\cap\bar\F(\bold 0))$.
\smallskip

For every $f\in\F$ consider the number $\e_f\in\{-1,0,1\}$ defined by the formula
$$
\e_f=\begin{cases}
1&\mbox{if $f\in\F_E^{<}$}\\
0&\mbox{if $f\in\F_E^{=}$}\\
-1&\mbox{if $f\in\F_E^{>}$}.
\end{cases}
$$
Taking into account that the subset $\mu_E\circ\bar\F_\partial(\partial E)\subset\mu_E(C_E)\subset[0,1]^\F$ is finite, it is easy to find a sequence of positive real numbers $(\alpha_f)_{f\in\F}$ such that the linear map
$$\lambda_E:[0,1]^\F\to\IR,\;\;\lambda_E:(x_f)_{f\in F}\mapsto \sum_{f\in F}\e_f\alpha_fx_f$$ is injective on the set $\mu_E\circ\F_\partial(\partial E)$.

\begin{claim}\label{cl:l7b} For each pair $\{a,b\}\in E$ the map $\lambda_E\circ\mu_E$ is injective on the set $\bar\F_\partial([a,b])$.
\end{claim}

\begin{proof} Assume that $\lambda_E\circ\mu_E(y)=\lambda_E\circ\mu_E(y')$ for some  points $y=(y_f)_{f\in\F}$ and $y'=(y'_f)_{f\in\F}$ in $\bar\F_\partial([a,b])$. Choose two points $x,x'\in [a,b]$ such that $y=\bar\F_\partial(x)$ and $y'=\bar\F_\partial (x')$. On the interval $[a,b]$ we consider the linear order such that $a<b$. We lose no generality assuming that $x<x'$ with respect to this order. The monotonicity of the maps $\bar f_h$, $f\in\F$, imply that
\begin{itemize}
\item $\bar f_h(x)\le \bar f_h(x')$  for each $f\in\F_E^{<}$;
\item $\bar f_h(x)\ge \bar f_h(x')$ for each $f\in\F_E^{>}$;
\item $\bar f_h(x)=\bar f_h(x')$ for each $f\in\F^{=}_E$.
\end{itemize}
Taking into account these inequalities, the increasing property of the maps $\mu_C$ and the choice of the numbers $\e_f$, $f\in\F$, we conclude that
$$\e_f\cdot\mu_{C_E^f}(\bar f_h(x))\le \e_f\cdot\mu_{C_E^f}(\bar f_h(x'))\mbox{ \ for all $f\in\F$}.$$ Consequently,
$$
\begin{aligned}
&\lambda_E\circ\mu_E(y)=\lambda_E\circ\mu_E\circ\bar\F_\partial(x)=
\sum_{f\in\F}\alpha_f\e_f\cdot\mu_{C_E^f}(\bar f_h(x))\le\\
&\le \sum_{f\in\F}\alpha_f\e_f\cdot\mu_{C_E^f}(\bar f_h(x')) =\lambda_E\circ\mu_E\circ\bar\F_\partial(x')=\lambda_E\circ\mu_E(y').
\end{aligned}
$$
Taking into account that $\lambda_E\circ\mu_E(y)=\lambda_E\circ\mu_E(y')$, we conclude that $y_f=\bar f_h(x)=\bar f_h(x')=y_f'$ for all $f\in\F^{\ne}_E$. For each $f\in\F^{=}_E$ the function $\bar f_h|[a,b]$ is constant and hence $y_f=\bar f_h(x)=\bar f_h(x')=y'_f$. Consequently, $y=y'$, which means that the map $\lambda_E\circ\mu_E$ is injective on $\bar\F_\partial|[a,b]$.
\end{proof}

Let us recall that by $Y$ we denote the countable set of points $y\in\bar\IR^\F$ with infinite preimage $\bar\F_\partial^{-1}(y)$.
The following claim plays a crucial role in the proof of Lemma~\ref{l:even}.

\begin{claim}\label{cl:l7e} For any $y\in \IR\setminus\lambda_E\circ\mu_E(\bar\F_\partial(\partial E)\cup Y)$ the preimage $D_y=(\lambda_E\circ\mu_E\circ\bar\F_h|\bar E)^{-1}(y)$ is finite and contains even number of points.
\end{claim}

\begin{proof} If $D_y$ is empty, then there is nothing to prove. So, we assume that the set $D_y$ is not empty. Claims~\ref{cl:l7b} and \ref{cl:t5} imply that for any pair $\{a,b\}\in E$ the intersection $D_y\cap[a,b]$ contains at most one point. This point belongs to the interior $]a,b[$ as $y\notin\lambda_E\circ\mu_E\circ\bar\F_\partial(\partial E)$.

Since $\mu_E\circ\bar\F_\partial(\bar E)\subset\mu_E(\bar C_E)\subset [0,1]^\F\subset\IR^\F$, the preimage $W=(\mu_E\circ\bar\F_h)^{-1}(\IR^\F)$ is an open neighborhood of the set $\bar E$ in $\e\bar K\setminus\frac\e2K$ and
$$\eta_E=\lambda_E\circ\mu_E\circ\bar\F_h|W:W\to\IR$$ is a well-defined continuous map.


Observe that the formula
$$\mathcal R(u,v)=\lambda_E\circ\mu_E\circ\F(u,v)=\sum_{f\in\F}\alpha_f\e_f\cdot \mu_{C_E^f}(f(u,v))$$ determines a rational function on $\IR^2$.
We claim that this rational function is not constant. Indeed, the set $D_y$, being not empty, contains some point $x$ which lies in the interval $[a,b]$ for some pair $\{a,b\}\in E$. Since $\eta_E(x)=y\ne \eta_E(a)$, we can find $t>1$ such that
$ta,tx\in W$ and $\eta_E(ta)\ne \eta_E(tx)$. Now we see that
$$\mathcal R(h(ta))=\eta_E(ta)\ne\eta_E(tx)=\mathcal R(h(tx)),$$which means that the rational function $\mathcal R$ is not constant. So, it is legal to consider the plane algebraic curve $A_y=\mathcal R^{-1}(y)$.

The choice of the point $y$ guarantees that $y\notin\eta_E(\partial E)$. Since $W$ is an open neighborhood of $\bar E$ in $\e\bar K\setminus\frac\e2 K$ and $\lim_{m\to\infty}\delta_m=\frac\e2$, there is
a number $m\in\IN$ so large that:
\begin{itemize}
\item $[1,\frac{2\delta_m}{\e}]\cdot\bar E\subset W$,
\item $y\notin\eta_E([1,\frac{2\delta_m}{\e}]\cdot\partial E)$, and
\item the number $\e_m$ is $A_y$-small.
\end{itemize}

 Let $\mathcal A_y$ denote the family of connected components of the  set $A_y\cap\e_m\bar K_\circ$. Since $\e_m$ is $A_y$-small each set $A\in \A_y$ is an $\e_m$-elementary branch of the algebraic curve $A_y$. By $A^*\in\A_y$ we shall denote its conjugate $\e_m$-branch.

For each $\e_m$-elementary branch $A\in\A_y$ the preimage $B=h^{-1}(A)$ is a curve in the ``square annulus" $\delta_m \bar K\setminus \frac\e2K$. Let $B^*=h^{-1}(A^*)$ be the ``conjugate'' curve to $B=h^{-1}(A)$.
Now consider the family $\mathcal B_y=\{h^{-1}(A):A\in\A_y\}$ that decomposes into pairs of conjugate curves.

\begin{claim}\label{cl:l7c} For any curve $B\in\mathcal B_y$ and its closure $\bar B$ in $\IR^2$ the intersection $\bar B\cap\frac\e2 K_\partial$ is a non-empty convex subset of $\frac\e2 K_\partial$ such that $\bar B\cap \bar E\subset D_y$. If the intersection $\bar B\cap\bar E$ is not empty, then it is a singleton. \end{claim}

\begin{proof} We lose no generality assuming that the $\e_n$-elementary curve $A=h(B)\in\mathcal A_y$ is an east $\e_m$-elementary curve. The construction of the homeomorphism $h$ guarantees that the curve $B$ coincides with the graph of some continuous function defined on the interval $(\frac\e2,\delta_m]$.
This implies that the intersection $\bar B\cap \frac\e2K_\partial$ is a non-empty closed convex subset that lies in the east side $\{\frac\e2\}\times[-\frac\e2,\frac\e2]$ of the square $\frac\e2 K_\partial$. Taking into account that $\eta_E(W\cap B)=\{y\}$, we conclude that $\eta_E(W\cap \bar B)=\{y\}$ and hence $\bar B\cap\bar E\subset D_y$.

If $\bar B\cap \bar E$ is not empty, then it is a singleton because $\bar B\cap \frac\e2 K_\partial$ is convex, does not meet the set $\partial E$, and the intersection $\bar B\cap \bar E\subset D_y$ is finite.
\end{proof}

Let $$\mathcal B_y^E=\{B\in\mathcal B_y:\bar B\cap\bar E\ne\emptyset\}\mbox{ and }\mathcal A_y^E=\{h(B):B\in\mathcal B_y^E\}.$$ For every $B\in\mathcal B^E_y$ let $\pi(B)$ be the unique point of the intersection $\bar B\cap\bar E\subset D_y$.

The following claim completes the proof of Claim~\ref{cl:l7e} showing that the set $|D_y|=|\mathcal B^E_y|$ has even cardinality.
\end{proof}

\begin{claim}\label{cl:l7d}
\begin{enumerate}
\item For any pair $\{a,b\}\in E$ and $t\in(1,\frac{2\delta_m}{\e}]$ the segment $[ta,tb]$ meets at most one set $B\in\mathcal B_y^E$.
\item The function $\pi|\mathcal B_y^E:\mathcal B_y^E\to D_y$ is bijective.
\item For each curve $B\in\mathcal B_y^E$ its conjugate curve $B^*$ belongs to $\mathcal B_y^E$, so the cardinality $|\mathcal B_y^E|$ is even.
\end{enumerate}
\end{claim}

\begin{proof}
1. Assume that for some pair $\{a,b\}\in E$ and some $t\in(1,\frac{2\delta_m}{\e}]$ the segment $[ta,tb]$ meets two distinct curves $B,B'\in\mathcal B_y^E$ at some points $u,u'$, respectively. We lose no generality assuming the points $u,u'$ are ordered so that $[ta,u]\cap[u',tb]=\emptyset$.
Since $D\supset\frac12B_0$ there are two neighbor points $a_0,b_0\in\frac12B_0$ such that $[a,b]\subset[a_0,b_0]$ and $[a_0,a]\cap[b,b_0]=\emptyset$.

Now consider the points $h(u)$, $h(u')\in A_y\cap(t-1)\e\bar K_\circ$ and observe that
$[h(u),h(u')]\subset[h(ta),h(tb)]\subset [h(ta_0),h(tb_0)]$.
The property (13) of the set $B_0$ guarantees that for every function $f\in\F$ the restriction $f|[h(ta_0),h(tb_0)]$ either is constant or is injective. In particular,  for each function $f\in\F_{a,b}^{\ne}$ the restriction $f|[h(u),h(u')]$ is injective. Then the choice of the sign $\e_f$, guarantees that $\e_f f(h(u))<\e_ff(h(u'))$ and then $$y=\lambda_E\circ\mu_E\circ\F(h(u))<\lambda_E\circ\mu_E\circ\F(h(u'))=y,$$
which is  the desired contradiction.
\smallskip

2. First we check that the function $\pi|\mathcal B_y^E$ is injective. Assume that $\pi(B)=\pi(B')$ for two distinct curves $B,B'\in\mathcal B^E_y$. Let $x\in\pi(B)=\pi(B')\in\bar E\cap D_y$ and find an ordered pair $(a,b)\in\vec E$ such that $x\in {]a,b[}$. The connectedness of the curves $B$ and $B'$ implies that for some $t\in(1,2\delta_m]$ the segment $[ta,tb]$ intersects both curves $B$ and $B'$ which is forbidden by Claim~\ref{cl:l7d}(1).

Now we prove that the function $\pi|\mathcal B_y^E$ is surjective. Fix any point $x\in D_y$ and find an ordered pair $(a,b)\in\vec E$ such that $x\in {]a,b[}$. Since $y\notin\eta_E(\{a,b\})$, we conclude that $\bar\F_h(a)\ne\bar\F_h(x)\ne\bar\F_h(b)$. Then the choice of the signs $\e_f$, $f\in\F$, guarantees that $\eta_E(a)<\eta_E(x)=y<\eta_E(b)$. Choose a number $t\in(1,\frac{2\delta_m}{\e}]$ such that $$\eta_E(ta)<y<\eta_E(tb).$$  It follows that the point $y$ belongs to $\eta_E([ta,tb])$ and the segment $[ta,tb]$ meets the preimage $B=h^{-1}(A)$  of some $\e_m$-branch $A$ of the algebraic curve $A_y$. Taking into account that $A$ is an $\e_m$-elementary curve and the intervals $[a,ta]$ and $[b,tb]$ do not intersect $B$, we conclude that the curve $B$ has a limit point $\pi(B)$ in the singleton $[a,b]\cap D_y=\{x\}$.
\smallskip

3. Take any curve $B\in\mathcal B^E_y$ and consider its conjugate curve $B^*$. Choose any point $x^*\in\bar B\cap \frac\e2 K_\partial$. For the points $x=\pi(B)$ and $x^*$ find pairs $\{a,b\},\{a^*,b^*\}\in\mathcal N\{D\}$ such that $x\in[a,b]$ and $x^*\in[a^*,b^*]$. It follows from $B\in\mathcal B^E_y$ that the pair $\{a,b\}\in E$. We need to show that the pair $\{a^*,b^*\}$ also belongs to $E$, which means that $\{a^*,b^*\}$ is $\F$-coherent to $\{a,b\}$. This will follow from Claim~\ref{cl:D}(3) as soon as we check that $\bar \F_h(x)=\bar \F_h(x^*)$.

Since the curves $B$ and $B^*$ are conjugated, their images $A=h(B)$ and $A^*=h(B^*)$ are conjugated $\e_m$-branches of the algebraic curve $A_y$. Lemma~\ref{l:conjug} implies that
$$\lim_{A\ni z\to\mathbf 0}\F(z)=\lim_{A^*\ni z\to\mathbf 0}\F(z).$$
Using the continuity of the function $\bar\F_h$ at the points $x$ and $x^*$, we see that
$$
\begin{aligned}
\bar\F_h(x)&=\lim_{B\ni u\to x}\bar \F_h(u)=\lim_{B\ni u\to x}\F\circ h(u)=\lim_{A\ni v\to\mathbf 0}\F(v)=\\
&=\lim_{A^*\ni v\to\mathbf 0}\F(v)=\lim_{B^*\ni u\to x^*}\F\circ h(u)=\lim_{B^*\ni u\to x^*}\F_h(u)=\bar\F_h(x^*).
\end{aligned}
$$
\end{proof}

Now we can continue the proof of Lemma~\ref{l:even}.

Choose a finite subset $N_E\subset\IR$ such that
\begin{itemize}
\item the convex hull $\conv(N_E)$ of $N_E$ contains the compact subset $\eta_E(\bar E)$ of $\IR$;
\item $\eta_E(\partial E)\subset N_E$;
\item for any neighbor points $a,b$ of $\eta_E(\partial E)$ the interval $]a,b[$ has non-empty intersection with the set $N_E$.
\end{itemize}

Fix a continuous map $\varphi_E:\IR\to I_E\subset\frac\e2K_\partial$ such that
\begin{itemize}
\item $\varphi_E\circ \eta_E|\partial E=g\circ\bar\F_\partial|\partial E$;
\item for any neighbor points $a,b$ of the set $N_E$ the restriction $\varphi|[a,b]:[a,b]\to I_E$ is injective.
\end{itemize}
Finally,  put $$\tilde g_E=\varphi_E\circ\eta_E|\bar E:\bar E\to I_E\subset\tfrac\e2K_\partial.$$

Taking into account that $\tilde g_E(\bar E)\cup g\circ\bar\F_\partial(\bar E)\subset I_E$ and $\tilde g_E|\partial E=g\circ\bar\F_{\partial}|\partial E$, we see that the maps $\tilde g_E,g\circ\bar\F_\partial|\bar E:\bar E\to I_E$ are homotopic by a homotopy $h_E:\bar E\times[0,1]\to I_E$ such that
\begin{itemize}
\item $h_E(x,0)=\tilde g_E(x),\;h_E(x,1)=g\circ\bar\F_\partial(x)$ for all $x\in\bar E$ and
\item $h_E(x,t)=\tilde g_E(x)=g\circ\bar\F_\partial(x)$ for all $x\in \partial E$ and $t\in[0,1]$.
\end{itemize}

The maps $\tilde g_E$, $E\in\lrN\{D\}$, compose a  map $\tilde g:\frac\e2K_\partial\to\frac\e2K_\partial$ defined by $\tilde g|\bar E=\tilde g_E$ for $E\in\lrN\{D\}$.

Also, the homotopies $h_E:\bar E\times[0,1]\to I_E\subset \frac\e2 K_\partial$, $E\in\lrN\{D\}$, compose a homotopy $h:\frac\e2 K_\partial\times[0,1]\to\frac\e2 K_\partial$ between the maps $\tilde g$ and $g\circ\bar\F_\partial$.

The proof of Lemma~\ref{l:even} is finished by the following claim.

\begin{claim} The map $\tilde g$ is $\IZ_2$-trivial.
\end{claim}

\begin{proof} To show that the map $\tilde g$ is $\IZ_2$-trivial, we shall apply Lemma~\ref{l:degree}. Pick any point $y_0\in\frac\e2K_\partial$ which does not belong to the countable set
$$\bigcup\big\{\tilde g_E(\partial E)\cup \varphi(N_E)\cup\varphi_E\circ\lambda_E\circ\mu_E(Y\cap\bar C_E) :E\in\lrN\{D\}\big\}.$$
For every equivalence class $E\in\lrN\{D\}$ consider the set $Y_E=\varphi^{-1}_E(y_0)$ which is finite by the choice of the function $\varphi_E$. By Claim~\ref{cl:l7e}, for every point $y\in Y_E$ the preimage $D_y=(\lambda_E\circ\mu_E\circ\bar\F_h|\bar E)^{-1}(y)$ is finite and contains even number of points. Since $y_0\notin \tilde g(\partial E)$, we get $D_y\subset \bar E\setminus\partial E$. Then the preimage $\tilde g_E^{-1}(y_0)=\bigcup_{y\in Y_E}D_y$ lies in $\bar E\setminus\partial E$ and contains even number of points. Unifying these preimages, we conclude that the preimage
$$\tilde g^{-1}(y_0)=\bigcup\{\tilde g^{-1}_E(y_0):E\in\lrN\{D\}\}$$ has even cardinality and lies in the set $\frac\e2K_\partial\setminus D$.

It remains to check that each point $x\in\tilde g^{-1}(y_0)$ has a neighborhood $U_x\subset\frac\e2K_\partial$ such that the map $\tilde g|U_x$ is monotone. Find two neighbor points $a,b$ of the set $D$ such that $x\in{]a,b[}$.
Let $E$ be the $\F$-coherence class of the pair $\{a,b\}$. By Claims~\ref{cl:t6a} and \ref{cl:l7b}, the map $\lambda_E\circ\mu_E\circ\F_h|[a,b]$ is monotone. Now consider the point $y=\lambda_E\circ\mu_E\circ\bar\F_h(x)$ and observe that $y\notin N_E$ (as $\varphi_E(y)=y_0\notin \varphi_E(N_E)$). The choice of the function $\varphi_E$ guarantees that the point $y$ has a neighborhood $V_y\subset \IR\setminus N_E$ such that the restriction $\varphi_E|V_y$ is injective (and hence monotone). Then the neighborhood $U_x=(\lambda_E\circ\mu_E\circ\bar\F_h|[a,b])^{-1}(V_y)$ has the desired property: the restriction $\tilde g|U_x=\tilde g_E|U_x=\varphi_E\circ\lambda_E\circ\mu_E\circ\bar\F_h|U_x$ is monotone.
\end{proof}

\end{proof}
\end{proof}

\section{Inverse spectra}

In the proof of Theorems~\ref{main1} and \ref{main2} we shall widely use the technique of inverse spectra, see \cite{Eng}, \cite{Chi}. Formally speaking an {\em inverse spectrum} in a category $\mathcal C$ is a contravariant functor $\mathcal S:\Sigma\to \C$ from a directed partially ordered set $\Sigma$ to the category $\C$. A partially ordered set (briefly a poset) $\Sigma$ is called {\em directed} if for any elements $\alpha,\beta\in\Sigma$ there is an element $\gamma\in\Sigma$ such that $\gamma\ge\alpha$ and $\gamma\ge\beta$. Each poset $\Sigma$ can be identified with a category whose objects are elements of $\Sigma$ and two objects $\alpha,\beta\in\Sigma$ are linked by a single morphism $\alpha\to\beta$ if and only if $\alpha\le\beta$.

An inverse spectrum $\mathcal S:\Sigma\to\C$ can be written directly as the family $\{X_\alpha,p_\alpha^\beta,\Sigma\}$ consisting of objects $X_\alpha$ of the category $\C$, indexed by elements $\alpha$ of the poset $\Sigma$, and bonding morphisms $p_\alpha^\beta:X_\beta\to X_\alpha$ defined for any indices $\alpha\le\beta$ in $\Sigma$, so that for any indices $\alpha\le\beta\le\gamma$ in $\Sigma$ the following two conditions are satisfied:
\begin{itemize}
\item $p_\alpha^\gamma=p^\beta_\alpha\circ p_\beta^\gamma$ and
\item $p_\alpha^\alpha$ is the identity morphism of $X_\alpha$.
\end{itemize}

Inverse spectra over a poset $\Sigma$ in a category $\C$ form a category $\C_\Sigma$ whose morphisms are natural transformations of functors. In other words, for two inverse spectra $\mathcal S=\{X_\alpha,p_\alpha^\beta,\Sigma\}$ and $\mathcal S'=\{X'_\alpha,\pi_\alpha^\beta,\Sigma\}$ a morphism $f:\mathcal S\to\mathcal S'$ in $\C_\Sigma$ is a family of morphisms $f=\{f_\alpha:X_\alpha\to X_\alpha'\}_{\alpha\in\Sigma}$ of the category $\C$ such that for any indices $\alpha\le\beta$ in $\Sigma$ the following square is commutative:
$$\xymatrix{
X_\beta\ar[d]_{p^\beta_\alpha}\ar[r]^{f_\alpha}
&{X_\beta'}\ar[d]^{\pi^\beta_\alpha}\\
X_\alpha\ar[r]_{f_\alpha}&X_\alpha'
}
$$

There is a functor $(\cdot)_\Sigma:\C\to\C_\Sigma$ assigning to each object $X$ of $\C$ the inverse spectrum $X_\Sigma=\{X_\alpha,p_\alpha^\beta,\Sigma\}$ where $X_\alpha=X$ and $p_\alpha^\beta$ is the identity map of $X$ for all $\alpha\le\beta$ in $\Sigma$. To each morphism $f:X\to Y$ of the category $\C$ the functor $(\cdot)_\Sigma$ assigns the morphism $f_\Sigma=\{f_\alpha\}_{\alpha\in\Sigma}$ where $f_\alpha=f$ for all $\alpha\in\Sigma$.

For an inverse spectrum $\mathcal S:\Sigma\to\C$ its {\em limit} is a pair $(X,p)$ consisting of an object $X$ of $\C$ and a morphism $p=\{p_\alpha\}_{\alpha\in\Sigma}:X_\Sigma\to\mathcal S$ in the category $\C_\Sigma$ such that for any other pair $(Z,\pi)$ consisting of an object $Z$ of $\C$ and a morphism $\pi=\{\pi_\alpha\}_{\alpha\in\Sigma}:Z_\Sigma\to\mathcal S$ there is a unique morphism $f:Z\to X$ such that $\pi=p\circ f_\Sigma$.
This definition implies that a limit $(X,p)$ of $\mathcal S$ if exists, is unique up to the isomorphism. Because of that the space $X$ is denoted by $\lim\mathcal S$ and called {\em the limit} of the inverse spectrum $\mathcal S$.
\smallskip

In this paper we shall be mainly interested in inverse spectra in the category $\CompEpi$ of compact Hausdorff spaces and their continuous surjective maps.
In this case, each inverse spectrum $\mathcal S=\{X_\alpha,p_\alpha^\beta,\Sigma\}$ has a limit $(X,p)$ consisting of the closed subspace
$$X=\big\{(x_\alpha)_{\alpha\in\Sigma}\in\prod_{\alpha\in\Sigma}X_\alpha:
p^\beta_\alpha(x_\beta)=x_\alpha\mbox{ for all $\alpha\le \beta$ in $\Sigma$}\big\}$$
of the Tychonoff product $\prod_{\alpha\in\Sigma}X_\alpha$ and the morphism $p=(p_\alpha)_{\alpha\in\Sigma}:X_\Sigma\to\mathcal S$ where $p_\alpha:X\to X_\alpha$, $p_\alpha:(x_\alpha)_{\alpha\in\Sigma}\mapsto x_\alpha$, is the $\alpha$-th coordinate projection.

Using the technique of inverse spectra, we shall reduce the problem of investigation of the graphoid $\GGamma(\F)$ of an arbitrary family $\F\subset \IR(x_1,\dots,x_k)$ to studying the graphoids $\GGamma(\alpha)$ of finite subfamilies $\alpha$ of $\F$. Namely, given any family  $\F\subset\IR(x_1,\dots,x_k)$ of rational functions of $k$-variables, consider the set $\Sigma=[\F]^{<\w}$ of finite subsets of $\F$, partially ordered by the inclusion relation $\subset$. Endowed with this relation, $\Sigma=[\F]^{<\w}$ becomes a directed poset. For any elements $\alpha\le\beta$ of $\Sigma$ (which are finite subsets $\alpha\subset\beta$ of $\F$) we can consider the coordinate projection $p_\alpha^\beta:\GGamma(\beta)\to\GGamma(\alpha)$. In such a way we obtain the inverse spectrum $\mathcal S_\F=\{\GGamma(\alpha),p_\alpha^\beta,\Sigma\}$ consisting of graphoids of finite subfamilies of $\F$. For each finite subset $\alpha\in\Sigma$ of $\F$ the limit projection  $p_\alpha:\GGamma(\F)\to\GGamma(\alpha)$ coincides with the corresponding coordinate projection (we recall that $\GGamma(\F)=\bar\IR^2\times\bar\IR^\F$ while $\GGamma(\alpha)\subset\bar\IR^2\times\bar\IR^\alpha$).

The crucial fact that follows from the definition of $\GGamma(\F)$ is the following lemma:

\begin{lemma}\label{l:4.1} The graphoid $\GGamma(\F)$ together with the limit projections $p_\alpha:\GGamma(\F)\to\GGamma(\alpha)$, $\alpha\in\Sigma$, is the limit of the inverse spectrum $\mathcal S_\F=\{\GGamma(\alpha),p_\alpha^\beta,\Sigma\}$ consisting of graphoids $\GGamma(\alpha)$ of finite subfamilies $\alpha\subset\F$.
\end{lemma}

\section{Extension dimension of limit spaces of inverse spectra}

In this section we shall evaluate the extension dimension of limit spaces of inverse spectra in the category $\CompEpi$. This information will be then used in the proofs of Theorems~\ref{main1} and \ref{main2}.

We shall say that a topological space $Y$ is an {\em absolute neighborhood extensor for compact Hausdorff spaces} and write $Y\in \ANE(\Comp)$ if each map $f:A\to Y$ defined on a closed subspace $A$ of a compact Hausdorff space $X$ has a continuous extension $\bar f:N(A)\to Y$ defined on a neighborhood $N(A)$ of $A$ in $X$.

Let us recall that a space $X$ has extension dimension $\edim(X)\le Y$ if each map $f:A\to Y$ defined on a closed subspace $A$ of $X$ has a continuous extension $\bar f:X\to Y$.

The following lemma should be known in Extension Dimension Theory but we could not find a precise reference. So, we have decided to give a proof for convenience of the reader.

\begin{lemma}\label{l:5.1} Let $\big(X,(p_\alpha)\big)$ be a limit of an inverse spectrum $S=\{X_\alpha,p_\alpha^\beta,\Sigma\}$ in the category $\CompEpi$. The limit space $X$ has extension dimension $\edim(X)\le Y$ for some space $Y\in\ANE(\Comp)$ if and only if for any $\alpha\in\Sigma$ and a map $f_\alpha:A_\alpha\to Y$ defined on a closed subspace $A_\alpha$ of the space $X_\alpha$ there are an index $\beta\ge \alpha$ in $\Sigma$ and a map $\bar f_\beta:X_\beta\to Y$ that extends the map $f_\alpha\circ p^\beta_\alpha|A_\beta:A_\beta\to Y$ defined on the closed subset $A_\beta=(p^\beta_\alpha)^{-1}(A_\alpha)$ of $X_\beta$.
\end{lemma}

\begin{proof} First we prove the ``if'' part of the lemma.
To prove that $X$ has extension dimension $\edim(X)\le Y$, fix a map $f:A\to Y$ defined on a closed subset $A$ of the space $X$. Embed the space $X$ into a Tychonoff cube $[0,1]^\kappa$. Since $Y\in\ANE(\Comp)$, the map $f$ admits a continuous extension $\tilde f:O(A)\to Y$ defined on an open neighborhood $O(A)$ of $A$ in $[0,1]^\kappa$. Next, find a closed neighborhood $\tilde A\subset O(A)$ of $A$ in $[0,1]^\kappa$.

Let $\U$ be a cover of $[0,1]^\kappa$ by open convex subsets such that $$\St(\tilde A,\U):=\bigcup\big\{U\in\U:\tilde A\cap U\ne\emptyset\big\}\subset O(A).$$

\begin{claim}\label{cl:5.2} There is an index $\alpha\in\Sigma$ and a continuous map $r_\alpha:\tilde A_\alpha\to O(A)$ defined on the closed subset $\tilde A_\alpha=p_\alpha(\tilde A)$ of $X_\alpha$ such that the map $r_\alpha\circ p_\alpha|\tilde A$ is $\U$-near to the identity embedding $\tilde A\to O(A)$ in the sense that for each $x\in \tilde A$ the doubleton $\{x,r_\alpha\circ p_\alpha(x)\}$ lies in some set $U\in\U$.
\end{claim}

\begin{proof} Let $\V$ be an open cover of $[0,1]^\kappa$ that star-refines the cover $\U$ (the latter means that for every $V\in\V$ its $\V$-star $\St(V,\V)$ lies in some set $U\in\U$).

It is well-known that the topology of the limit space $X$ of the spectrum $\mathcal S$ is generated by the base consisting of the sets $p_\alpha^{-1}(U_\alpha)$ where $\alpha\in\Sigma$ and $U_\alpha$ is an open set in $X_\alpha$. Here $p_\alpha:X\to X_\alpha$ stands for the limit projection.

This fact allows us to find for every $z\in \tilde A$ an index $\alpha_z\in\Sigma$ and an open neighborhood $W_z\subset X_{\alpha_z}$ of $p_{\alpha_z}(z)$ such that the neighborhood $p^{-1}_{\alpha_z}(W_z)$ of $z$ lies in some set $V_z\in\V$. The open cover $\{p^{-1}_{\alpha_z}(W_z):z\in \tilde A\}$ of the compact subset $\tilde A$ admits a finite subcover $\{p^{-1}_{\alpha_z}(W_z):z\in F\}$. Here $F$ is a suitable finite subset of $\tilde A$. Since the index set $\Sigma$ is directed, there is an index $\alpha\in\Sigma$ such that $\alpha\ge\alpha_z$ for all $z\in F$. Changing the sets $W_z$ by $(p^\alpha_{\alpha_z})^{-1}(W_{z})$, we can assume that $\alpha_z=\alpha$ for all $z\in F$. Then $\W=\{W_z:z\in F\}$ is an open cover of the closed subset $\tilde A_\alpha=p_\alpha(\tilde A)$ of the compact space $X_\alpha$. Let $\{\lambda_z:\tilde A_\alpha\to[0,1]\}_{z\in F}$ be a partition of unity, subordinated to the cover $\W$ in the sense that $\lambda_z^{-1}\big(]0,1]\big)\subset W_z$ for all $z\in F$.

Consider the map $r_\alpha:\tilde A_\alpha\to[0,1]^\kappa$ defined by
$$r_\alpha(x)=\sum_{z\in F}\lambda_z(x)\cdot z.$$
We claim that this map has the required property: $r_\alpha\circ p_\alpha|\tilde A$ is $\U$-near to the identity embedding $\tilde A\to O(A)$.

Given any point $x\in\tilde A$, consider the finite set $E=\{z\in F:\lambda_z(p_\alpha(x))>0\}$. It follows that $$r_\alpha(p_\alpha(x))=\sum_{z\in E}\lambda_z(p_\alpha(x))\cdot z.$$
Observe that for every $z\in E$ we get
$$x\in p^{-1}_\alpha(p_\alpha(x))\subset p_\alpha^{-1}\big(\lambda_z^{-1}\big(]0,1]\big)\big)\subset p^{-1}_\alpha(W_z)\subset V_z$$and hence $$E\cup\{x\}\subset\bigcup_{z\in E}V_z\subset\St(x,\V)\subset U$$for some open convex set $U\in\U$.
The convexity of the set $U$ guarantees that this set contains the following convex combination:
$$r_\alpha(p_\alpha(x))=\sum_{z\in E}\lambda_z(p_\alpha(x))\cdot z.$$
\end{proof}

Claim~\ref{cl:5.2} implies that $$r_\alpha(\tilde A_\alpha)=r_\alpha\circ p_\alpha(\tilde A)\subset\St(\tilde A,\U)\subset O(A),$$ so the composition $f_\alpha=\tilde f\circ r_\alpha:\tilde A_\alpha\to Y$ is a well-defined continuous map. By our assumption, there is an index $\beta\ge\alpha$ and a continuous map $\bar f_\beta:X_\beta\to Y$ that extends the map $f_\beta=f_\alpha\circ p_\alpha^\beta|\tilde A_\beta$, where $\tilde A_\beta=(p_\alpha^\beta)^{-1}(\tilde A_\alpha)\supset p_\beta(\tilde A)$.
Observe that
$$\bar f_\beta\circ p_\beta|\tilde A=f_\beta\circ p_\beta|\tilde A=f_\alpha\circ p_\alpha^\beta\circ p_\beta|\tilde A=f_\alpha\circ p_\alpha|\tilde A=\tilde f\circ r_\alpha\circ p_\alpha|\tilde A.$$
Using the Urysohn Lemma, choose a continuous function $\xi:X\to[0,1]$ such that $\xi(A)\subset\{1\}$ and $X\setminus \tilde A\subset \xi^{-1}(0)$.

Claim~\ref{cl:5.2} implies that for every $x\in \tilde A$ the convex combination $\xi(x)x+(1-\xi(x))r_\alpha(p_\alpha(x))$ lies in $\St(\tilde A,\U)\subset O(A)$ so, the function $\bar f:X\to Y$,
$$\bar f:x\mapsto\begin{cases}
\tilde f\big(\xi(x)x+(1-\xi(x))r_\alpha\circ p_\alpha(x)\big)&\mbox{if $x\in \tilde A$}\\
\bar f_\beta\circ p_\beta(x)&\mbox{if $x\in \xi^{-1}(0)$,}\\
\bar f_\beta\circ p_\beta(x)=\tilde f\circ r_\alpha\circ p_\alpha(x)&\mbox{if $x\in \xi^{-1}(0)\cap \tilde A$,}
\end{cases}
$$is a well-defined continuous extension of the map $f=\tilde f|A$, witnessing that $\edim(X)\le Y$.
\smallskip

Now we prove the ``only if'' part of the lemma. Assume that $\edim(X)\le Y$. Fix an index $\alpha\in\Sigma$ and a continuous map $f_\alpha:A_\alpha\to Y$ defined on a closed subset $A_\alpha$ of $X_\alpha$. We need to find an index $\beta\ge\alpha$ in $\Sigma$ and a continuous map $\bar f_\beta:X_\beta\to Y$ that extends the map $f_\alpha\circ p_\alpha^\beta|A_\beta$ defined on the subset $A_\beta=(p_\alpha^\beta)^{-1}(A_\alpha)$ of $X_\beta$.

Since $Y\in\ANE(\Comp)$, the map $f_\alpha$ admits a continuous extension $\tilde f_\alpha:\tilde A_\alpha\to Y$ defined on a closed neighborhood $\tilde A_\alpha$ of $A_\alpha$ in $X_\alpha$. Then $\tilde A=p_\alpha^{-1}(\tilde A_\alpha)$ is a closed neighborhood of the closed set $A=p_\alpha^{-1}(A_\alpha)$ in $X$. Since $\edim(X)\le Y$, the map $\tilde f_\alpha\circ p_\alpha|\tilde A$ has a continuous extension $\bar f:X\to Y$.

Embed the compact Hausdorff space $K=\bar f(X)\subset Y$ in a Tychonoff cube $[0,1]^\kappa$ of a suitable weight $\kappa$. Since $Y\in\ANE(\Comp)$, the identity embedding $K\to Y$ admits a continuous extension $\psi:O(K)\to Y$ defined on an open neighborhood $O(K)$ of $K$ in $[0,1]^\kappa$. Let $\U$ be a cover of $[0,1]^\kappa$ by open convex subsets such that $\St(K,\U)\subset O(K)$. Repeating the argument of Claim~\ref{cl:5.2}, we can find an index $\beta\ge\alpha$ in $\Sigma$ and a continuous map $f_\beta:X_\beta\to [0,1]^\kappa$ such that the composition $f_\beta\circ p_\beta$ is $\U$-near to the map $\bar f:X\to K\subset [0,1]^\kappa$.

Consider the closed neighborhood $\tilde A_\beta=(p_\alpha^\beta)^{-1}(\tilde A_\alpha)\supset p_\beta(\tilde A)$ of the set $A_\beta=(p_\alpha^\beta)^{-1}(A_\alpha)$ in the space $X_\beta$. Using the Urysohn Lemma, choose a continuous function $\xi:X_\beta\to[0,1]$ such that $A_\beta\subset \xi^{-1}(1)$ and $X_\beta\setminus \tilde A_\beta\subset\xi^{-1}(0)$.

Given any point $y\in \tilde A_\beta$, choose a point $x\in p_\beta^{-1}(y)\subset \tilde A$ (which exists by the surjectivity of the limit projection $p_\beta$), and observe that
$$\{\tilde f_\alpha\circ p_\alpha^\beta(y),f_\beta(y)\}=\{\tilde f_\alpha\circ  p_\alpha(x),f_\beta\circ  p_\beta(x)\}=\{\bar f(x),f_\beta\circ p_\beta(x)\}\subset U\subset \St(K,\U)\subset O(K)$$for some set $U\in\U$ according to the choice of the map $f_\beta$.
Then the convex combination $\xi(x)\tilde f_\alpha(p_\alpha^\beta(x))+(1-\xi(x))f_\beta(x)$ also belongs to $U\subset O(K)$, which implies that the map
$\bar f_\beta:X_\beta\to Y$,
$$\bar f_\beta(x)=\begin{cases}
\psi\big(\xi(x)\tilde f_\alpha\circ p_\alpha^\beta(x)+(1-\xi(x))f_\beta(x)\big)&\mbox{if $x\in \tilde A_\beta$}\\
\psi(f_\beta(x))&\mbox{if $x\in\xi^{-1}(0)$}\\
\psi(f_\beta(x))=\psi\circ \tilde f_\alpha\circ p_\alpha^\beta(x))&\mbox{if $x\in\tilde A_\beta\cap \xi^{-1}(0)$}
\end{cases}
$$is a well-defined continuous extension of the map $\psi\circ \tilde f_\alpha\circ p_\alpha^\beta|A_\beta=f_\alpha\circ p_\alpha^\beta|A_\beta$.
\end{proof}

Lemma~\ref{l:5.1} implies the following known fact on preservation of extension dimension by inverse limits.

\begin{corollary}\label{c:5.3} Let $\big(X,(p_\alpha)\big)$ be a limit of an inverse spectrum $S=\{X_\alpha,p_\alpha^\beta,\Sigma\}$ in the category $\CompEpi$. The limit space $X$ has extension dimension $\edim(X)\le Y$ for some space $Y\in\ANE(\Comp)$ provided that $\edim(X_\alpha)\le Y$ for all $\alpha\in\Sigma$.
\end{corollary}

By \cite[3.2.9]{End}, a compact Hausdorff space $X$ has covering dimension $\dim X\le n$ if and only if $\edim(X)\le S^n$ where $S^n$ denotes the $n$-dimensional sphere. This fact combined with Corollary~\ref{c:5.3} yields the following well-known fact \cite[3.4.11]{End}:

\begin{corollary}\label{c:5.4} Let $\big(X,(p_\alpha)\big)$ be a limit of an inverse spectrum $S=\{X_\alpha,p_\alpha^\beta,\Sigma\}$ in the category $\CompEpi$. The limit space $X$ has dimension $\dim(X)\le n$ for some $n\in\w$ provided that $\dim(X_\alpha)\le n$ for all $\alpha\in\Sigma$.
\end{corollary}

\section{Proof of Theorem~\ref{main1}}\label{s:t1}

In this section we present a proof of Theorem~\ref{main1}. Given any non-empty family of rational functions $\F\subset\IR(x,y)$ we need to prove that the graphoid $\GGamma(\F)$ has dimension $\dim(\bar \Gamma(\F))=2$.

\begin{lemma} The graphoid $\bar\Gamma(\F)$  has dimension $\dim(\bar\Gamma(\F))\le 2$.
\end{lemma}

\begin{proof}  By Lemma~\ref{l:4.1}, the graphoid $\GGamma(\F)$ is homeomorphic to the limit space of the inverse spectrum $\mathcal S_\F=\{\GGamma(\alpha),p_\alpha^\beta,[\F]^{<\w}\}$ that consists of graphoids $\GGamma(\alpha)$ of finite subfamilies $\alpha\subset\F$.
Now Corollary~\ref{c:5.4} will imply that $\dim\GGamma(\F)\le 2$ as soon as we check that $\dim\GGamma(\G)\le 2$ for any finite subfamily $\G\subset \F$.

Since $\G$ is finite,  the set $\dom(\G)=\bigcap_{f\in\G}\dom(f)$ is cofinite in $\bar \IR^2$.
Identify the family $\G$ with the partial continuous function $$\G:\dom(\G)\to\bar\IR^\G,\;\;\G:x\mapsto(f(x))_{f\in\G}$$and let $\bar\G$ be the graphoid extension of $\G$. Then $\bar \Gamma(\G)=\Gamma(\bar\G)$ and hence
$$\bar\Gamma(\G)=\Gamma(\bar\G)=\Gamma(\G)\cup
\bigcup\big\{\{z\}\times\bar \G(z):z\in\bar\IR^2\setminus\dom(\G)\big\}.$$

Theorem~\ref{local}(6) implies that for every point $z\in\dom(\G)$ the set $\bar \G(z)$ has dimension $\dim(\bar\G(z))\le 1$.
Since the graph $\Gamma(\G)$ is homeomorphic to the cofinite set $\bar\IR^2\setminus\dom(\G)$, it has dimension $\dim(\Gamma(\G))\le\dim(\bar\IR^2)=2$.
Now Theorem of Sum \cite[1.5.3]{End} implies that
$$\dim \bar\Gamma(\G)\le\sup\{\dim(\Gamma(\G)),\dim\bar\G(z): z\in\bar\IR^2\setminus\dom(\G)\}\le 2.$$
\end{proof}

\begin{lemma}\label{l:6.2} $\dim \GGamma(\F)\ge 2$.
\end{lemma}

\begin{proof} Since $\dim\GGamma(\F)\le 1$ if and only if $\edim(\GGamma(\F))\le K_\partial$, it suffices to check that $\edim(\GGamma(\F))\not\le K_\partial$. To prove this fact, we shall apply Lemma~\ref{l:5.1}. By Lemma~\ref{l:4.1}, the graphoid $\GGamma(\F)$ is the limit of the spectrum $\mathcal S_\F=\{\GGamma(\alpha),p_\alpha^\beta,[\F]^{<\w}\}$. The smallest element of the poset $[\F]^{<\w}$ is the empty set. Its graphoid $\GGamma(\emptyset)$ can be identified with the torus $\bar\IR^2$.

Let $A_\emptyset=2\bar K\setminus K$ where $K=(-1,1)^2$ is the open square in the plane $\IR^2$ endowed with the max-norm
$$\|(x,y)\|=\max\{|x|,|y|\}.$$

Consider the map $f_\emptyset:A_\emptyset\to K_\partial$, $f:(x,y)\mapsto\frac{(x,y)}{\|(x,y)\|}$. Assuming that $\edim \GGamma(\F)\le K_\partial$, and applying Lemma~\ref{l:5.1}, we can find a finite subset $\beta\subset\F$ and a continuous map $f_\beta:\GGamma(\beta)\to K_\partial$ that extends the map $f_\emptyset\circ p_\emptyset^\beta|A_\beta:A_\beta\to K_\partial$ where $A_\beta=(p_\emptyset^\beta)^{-1}(A_\emptyset)$. The finite family $\beta\subset\F$ thought as a partial function $\beta:\dom(\beta)\to\bar\IR^\beta$ is defined on a cofinite subset $\dom(\beta)$ of $\bar\IR^2$. So, we can find a real number $t\in[1,2]$ such that $tK_\partial \subset\dom(\beta)$. Consider the finite set $Z=t\bar K\setminus\dom(\beta)$ in $tK$.

Using Theorem~\ref{local}, we can find $\e>0$ so small that
\begin{enumerate}
\item $tK_\partial\cap \bar B(Z,\e)=\emptyset$;
\item $\bar B(z,\e)\cap\bar B(z',\e)=\emptyset$ for any distinct points $z,z'\in Z$;
\item there is a homeomorphism $h:t\bar K\setminus \bar B(Z,\frac\e2)\to t\bar K\setminus Z$ such that
\begin{enumerate}
\item $h$ is identity on the set $t\bar K\setminus B(Z,\e)$,
\item $h$ has continuous extension $\bar h:t\bar K\setminus B(Z,\frac\e2)\to t\bar K$ such that $\bar h^{-1}(z)=S(z,\frac\e2)$ for every $z\in Z$;
\item the composition $\beta\circ h:t\bar K\setminus \bar B(Z,\frac\e2)\to \bar\IR^\beta$ has a continuous extension $\bar\beta_h:t\bar K\setminus  B(Z,\frac\e2)\to \bar\IR^\beta$;
\item for every $z\in Z$ and any map $\varphi_z:\bar\beta(z)\to K_\partial$ the composition $\varphi_z\circ \bar \beta_h|S(z,\frac\e2):S(z,\frac\e2)\to K_\partial$ is $\IZ_2$-trivial.
\end{enumerate}
\end{enumerate}
It follows that the map $$\psi=(\bar h,\bar\beta_h):t\bar K\setminus B(Z,\tfrac\e2)\to \GGamma(\beta),\;\;\psi:z\mapsto (\bar h(z),\bar\beta_h(z)),$$is continuous and for every $z\in Z$ the map $f_\beta\circ\psi|S(z,\frac\e2):S(z,\frac\e2)\to K_\partial$ is $\IZ_2$-trivial. Then by Lemma~\ref{l:2.5}, the map $\psi|t K_\partial:t K_\partial\to K_\partial$ also is $\IZ_2$-trivial, which is impossible as this map is a homeomorphism, which induces an isomorphism of the homology groups $H_1(tK_\partial;\IZ_2)$ and $H_1(K_\partial;\IZ_2)$. This contradiction completes the proof of Lemma~\ref{l:6.2}.
\end{proof}

\section{Proof of Theorem~\ref{main2}}\label{s:t2}

Assume that $\F\subset\IR(x,y)$ is a family of rational functions, containing a family of linear fractional transformations
$$\left\{\frac{x-a}{y-b}:(a,b)\in D\right\}$$
for some dense subset $D$ of $\IR^2$.

By Theorem~\ref{main1}, $\dim(\GGamma(\F))=2$. By Alexandroff Theorem \cite[1.4]{Dra},  $\dim_\IZ(X)=\dim(X)$ for each finite-dimensional compact Hausdorff space $X$. Consequently, $\dim_\IZ(\GGamma(\F))=\dim(\GGamma(\F))=2$.

Now let $G$ be a non-trivial 2-divisible abelian group. We need to show that $\dim_G(\GGamma(\F))=1$. To see that $\dim_G(\GGamma(\F))>0$, take any Eilenberg-MacLane complex $K(G,0)$, for example, take the group $G$ endowed with the discrete topology. Since the space $\GGamma(\F)$ is connected, any injective map $f:A\to G$ defined on a doubleton $A=\{a,b\}\subset\GGamma(\F)$ has no continuous extension $\bar f:\GGamma(\F)\to G$, which means that $\edim(\GGamma(\F))\not\le G$ and $\dim_G(\GGamma(\F))\not\le 0$.

The inequality $\dim_G(\GGamma(\F))\le 1$, which is equivalent to $\edim(\GGamma(\F))\le K(G,1)$, follows from the subsequent a bit more general result:

\begin{lemma} $\edim(\GGamma(\F))\le Y$ for any path-connected space $Y\in\ANE(\Comp)$ with 2-divisible fundamental group $\pi_1(Y)$.
\end{lemma}

\begin{proof} To show that $\edim(\GGamma(\F))\le Y$ we shall apply Lemma~\ref{l:5.1}. By Lemma~\ref{l:4.1}, the graphoid $\GGamma(\F)$ is homeomorphic to the limit space of the inverse spectrum $\mathcal S_\F=\{\GGamma(\alpha),p_\alpha^\beta,[\F]^{<\w}\}$. Given a finite subset $\alpha\in[\F]^{<\w}$ and a map $f_\alpha:A_\alpha\to Y$ defined on a closed subset $A_\alpha$ of the graphoid $\GGamma(\alpha)$, we need to find a finite subset $\beta\supset\alpha$ of $\F$ and a continuous function $\bar f_\beta:\GGamma(\beta)\to Y$ that extends the map $f_\alpha\circ p_\alpha^\beta|A_\beta$ defined on the set $A_\beta=(p_\alpha^\beta)^{-1}(A_\alpha)$.

We can think of the family $\alpha\subset\F$ as a partial function $\alpha:\dom(\alpha)\to\bar\IR^\alpha$ defined on the cofinite set $\dom(\alpha)$ in $\bar\IR^\alpha$.
Let $\bar\alpha:\bar\IR^2\setmap \bar\IR^\alpha$ be the graphoid extension of $\alpha$. Its graph $\Gamma(\bar\alpha)$ coincides with the graphoid $\GGamma(\alpha)$ of $\alpha$.

By Theorem~\ref{local}(6), for every point $z$ of the finite set $Z=\{(\infty,\infty)\}\cup\big(\bar\IR^2\setminus\dom(\alpha)\big)$ the image $\bar\alpha(z)$ is a singleton or a finite union of arcs. Consequently, the set $\Gamma(\bar\alpha|Z)=\bigcup_{z\in Z}\{z\}\times\bar\alpha(z)$ is a finite union of singletons or arcs. Using the path-connectedness of the space $Y\in\ANE(\Comp)$, we can extend the map $f_\alpha$ to a continuous map $f_\alpha':A_\alpha\cup\Gamma(\bar\alpha|Z)\to Y$.

Since $Y\in\ANE(\Comp)$, the map $f_\alpha':A_\alpha\cup\Gamma(\bar\alpha|Z)\to Y$ has a continuous extension $\tilde f_\alpha:\tilde A_\alpha\to Y$ defined on a closed neighborhood $\tilde A_\alpha$ of the set $A_\alpha\cup\Gamma(\bar\alpha|Z)$ in $\GGamma(\alpha)$. The boundary $\partial \tilde A_\alpha$ of $\tilde A_\alpha$ in $\GGamma(\alpha)$ is a compact subset of $\Gamma(\bar\alpha)\setminus\Gamma(\bar\alpha|Z)\subset\Gamma(\alpha)$.
The projection $p_\emptyset^\alpha:\Gamma(\bar\alpha)\to\bar\IR^2$ maps homeomorphically the graph $\Gamma(\alpha)$ onto the  cofinite subset $\dom(\alpha)$ of the torus $\bar \IR^2$.

Replacing $\tilde A_\alpha$ by a smaller (more regular) neighborhood, if necessary, we can assume that the boundary $\partial \tilde A_\alpha$ is a topological graph, that is, a finite union of arcs that are disjoint or meet by their end-points. Adding to $\tilde A_\alpha$ a finite union of arcs, we can enlarge the set $\tilde A_\alpha$ to a closed set $\bar A_\alpha\subset\GGamma(\alpha)$ whose boundary $\partial \bar A_\alpha$ is a topological graph such that
\begin{itemize}
\item the family $\C$ of connected components of $\GGamma(\alpha)\setminus\bar A_\alpha$ is finite and
\item for each connected component $C\in\C$ the closure $\bar C$ is homeomorphic to the closed square $\bar K=[-1,1]^2$.
\end{itemize}

Using the path-connectedness of the space $Y\in\ANE(\Comp)$, we can extend the map $\tilde f_\alpha$ to a continuous map $\bar f_\alpha:\bar A_\alpha\to Y$.

For every connected component $C\in\C$ use the density of the set $D$ in $\IR^2$ and find a point $(a_C,b_C)\in D\cap C$. Now consider the finite subfamily
$$\beta=\alpha\cup\Big\{\frac{x-a_C}{y-b_C}:C\in\C\Big\}\subset\F,$$
which determines a partial continuous function $\beta:\dom(\beta)\to\bar\IR^\beta$ defined on the cofinite set $\dom(\alpha)\setminus(\{(\infty,\infty)\}\cup\{(a_C,b_C):C\in\C\})$ of $\bar\IR^2$.

We claim that there is a continuous function $\bar f_\beta:\GGamma(\beta)\to Y$ that extends the map $f_\alpha\circ p_\alpha^\beta|A_\beta$ defined on the set $A_\beta=(p_\alpha^\beta)^{-1}(A_\alpha)$.

Put $\bar A_\beta=(p_\alpha^\beta)^{-1}(\bar A_\alpha)$ and observe that the complement $\GGamma(\beta)\setminus \bar A_\beta$ is the union of connected components $C^\beta=(p_\alpha^\beta)^{-1}(C)$, $C\in\C$, which are graphoids of the rational functions $\frac{x-a_C}{y-b_C}$ restricted to
the open 2-disks $p^\alpha_\emptyset(C)$. Such graphoids are homeomorphic to the open M\"obius band.

For every $C\in\C$ the closure $\bar C^\beta$ of $C^\beta$ in $\GGamma(\beta)$, being homeomorphic to the closed M\"obius band, is homeomorphic to the quotient space of the ``square annulus'' $\bar K\setminus \frac12K$ by the equivalence relation that identifies the pairs of opposite points on the inner boundary square $\frac12K_\partial$. Let $q_C:\bar K\setminus \frac12K\to \bar C^\beta$ be the corresponding quotient map.

Fix a continuous map $\sigma:[0,1]\to K_\partial$ such that
\begin{itemize}
\item $\sigma(0)=\sigma(1)$,
\item $\sigma\mid[0,1):[0,1)\to K_\partial$ is bijective,
\item for any $t\in[0,\frac12]$ the points $\sigma(t+\frac12)=-\sigma(t)$.
\end{itemize}
The map $\gamma_C=\bar f_\alpha\circ p_\alpha^\beta\circ q_C\circ \sigma:[0,1]\to Y$ determines a loop in $Y$, whose equivalence class is an element of the fundamental group $\pi_1(Y)$ of $Y$. Since the group $\pi_1(Y)$ is 2-divisible, there is a loop $\delta_C:[0,1]\to Y$ such whose square $\delta_C^2:[0,1]\to Y$,
$$\delta^2_C:t\mapsto\begin{cases}
\delta_C(2t)&\mbox{if $0\le t\le\frac12$,}\\
\delta_C(2t-1)&\mbox{if $\frac12\le t\le 1$,}
\end{cases}
$$is homotopic to the loop $\gamma_C$ by a homotopy that does not move the points $0$ and $1$.

Now consider the loop $\tilde \gamma_C=\bar f_\alpha\circ p_\alpha^\beta\circ q_C|K_\partial:K_\partial\to Y$ and observe that $\gamma_C=\tilde\gamma_C\circ\sigma$. Let  $\tilde\delta^2_C:\frac12K_\partial\to Y$ be a unique map such that $\tilde\delta^2_C\circ\frac12\sigma=\delta^2_C$. Here $\frac12\sigma:[0,1]\to\frac12K_\partial$ is the loop assigning to each $t\in[0,1]$ the point $\frac12\sigma(t)$ of the square $\frac12K_\partial$.
The homotopy between the loops $\gamma_C$ and $\delta^2_C$ allows us to find a continuous map $\tilde h_C:\bar K\setminus\frac12 K\to Y$ such that $\tilde h_C|K_\partial=\tilde \gamma_C$ and $\tilde h_C|\frac12K_\partial=\tilde\delta^2_C$.

The definition of $\sigma$ and $\delta^2_C$ guarantees that $\tilde \delta^2_C(x)=\tilde\delta^2_C(-x)$ for any point $x\in\frac12K_\partial$. Hence there is a unique continuous map $h_C:\bar C^\beta\to Y$ such that $\tilde h_C=h_C\circ q_C$. It follows from $\tilde h_C|K_\partial=\tilde\gamma_C$ that $h_C|\partial \bar C^\beta=\bar f_\alpha\circ p^\beta_\alpha|\partial\bar C^\beta$. This implies that the map $f_\beta:\GGamma(\beta)\to Y$,
$$f_\beta(x)=\begin{cases}
\bar f_\alpha\circ p_\alpha^\beta(x)&\mbox{if $x\in \bar A_\beta$,}\\
h_C(x)&\mbox{if $x\in\bar C^\beta$ for some $C\in\C$,}
\end{cases}
$$is a well-defined continuous extension of the map $f_\alpha\circ p_\alpha^\beta|A_\beta$.
\end{proof}

\section{Some Open Problems}

In light of Theorem~\ref{main2} the following problem arises naturally:

\begin{problem}\label{pr1} Has the graphoid $\GGamma(\F)$ of any family $\F\subset\IR(x,y)$ the cohomological dimension $\dim_G(\GGamma(\F))=2$ for any abelian group $G$ that is not 2-divisible?
\end{problem}

The answer to this problem is affirmative if the following problem has an affirmative answer.

\begin{problem}\label{pr2} Let $\F\subset\IR(x,y)$ be a finite family, $\bar\F:\bar\IR^2\setmap\bar\IR^\F$ be its graphoid extension, and $z\in\bar\IR^2$ be an arbitrary point. Is $\bar \F(z)$ a singleton or a finite union of analytic arcs in $\bar\IR^\F$.
\end{problem}

An arc $A$ in $\bar\IR^n$ is called {\em analytic} if $A=\vec\alpha([0,1])$ for some vector function $\vec\alpha:[0,1]\to\bar\IR^n$ that has analytic coordinate functions $\alpha_1,\dots,\alpha_n:[0,1]\to\bar\IR$. Here we identify the projective line $\bar\IR$ with the unit circle on plane via the stereographic projection.

In case of positive answer to Problem~\ref{pr2} the proof of the inequality $\dim(\GGamma(\F))\ge 2$ can be much simplified (Lemma 3 with its extremely long proof will be not required).

\end{document}